\documentclass[reqno,11pt]{amsart}
\usepackage[margin=1in]{geometry}
\usepackage{amssymb, amsthm}
\usepackage[numbers]{natbib}
\usepackage{mathtools}
\usepackage[dvipsnames]{xcolor}
\usepackage{graphbox}
\usepackage[all]{xy}
\usepackage{enumerate}
\input epsf
\theoremstyle{plain}
\newtheorem{thm}{Theorem}[section]

\newtheorem{cor}[thm]{Corollary}
\newtheorem{prop}[thm]{Proposition}
\newtheorem{conj}[thm]{Conjecture}
\newtheorem{theoremalph}{Theorem}

\theoremstyle{definition}
\newtheorem{exl}[thm]{Example}
\newtheorem{defn}[thm]{Definition}
\newtheorem{quest}[thm]{Question}
\newtheorem{prob}[thm]{Problem}
\newtheorem{rem}[thm]{Remark}

\usepackage{setspace}

\DeclareMathOperator{\lk}{lk}

\DeclareMathOperator{\Sig}{\Sigma}

\DeclareMathOperator{\C}{\mathcal{C}}
\DeclareMathOperator{\N}{\mathbb{N}}
\DeclareMathOperator{\Z}{\mathbb{Z}}

\DeclareMathOperator{\Q}{\mathbb{Q}}

\DeclareMathOperator{\Wh}{Wh}
\DeclareMathOperator{\M}{M}
\DeclareMathOperator{\Bl}{Bl}

\title{Homomorphism obstructions for satellite maps}
\author{Allison N.\ Miller}
\address{Department of Mathematics, Rice University, Houston, TX, United States}
\email{allison.miller@rice.edu}


\begin{document}
\maketitle
\begin{abstract}
A knot in a solid torus defines a map on the set of (smooth or topological) concordance classes of knots in $S^3$.
This set admits a group structure, but a conjecture of Hedden suggests that satellite maps never induce interesting homomorphisms: we give new evidence for this conjecture in both categories.  First, we use Casson-Gordon signatures to give the first obstruction to a slice pattern  inducing a homomorphism on the topological concordance group,  constructing examples with every winding number besides $\pm 1$. 
We then provide subtle examples of satellite maps which map arbitrarily deep into the $n$-solvable filtration of \cite{Cochran-Orr-Teichner:1999-1}, act like homomorphisms on arbitrary finite sets of knots, and yet which still do not induce homomorphisms. 
 Finally, we verify Hedden's conjecture in the smooth category for all but one small crossing number satellite operator. \end{abstract}

\section{Introduction}

A knot $P$ in the solid torus $S^1 \times D^2$ defines a function on the set of knots in $S^3$ by the well-known satellite construction: 
given a knot $K$, let $i_K \colon S^1 \times D^2 \to \nu(K) \subseteq S^3$ be an identification of the standard solid torus with a 0-framed tubular neighborhood of $K$ and define $P(K)$ to be $i_K(P)$. 

\begin{figure}[h!]
$\begin{array}{ccc}
\begin{array}{c} \includegraphics[height=2.8cm]{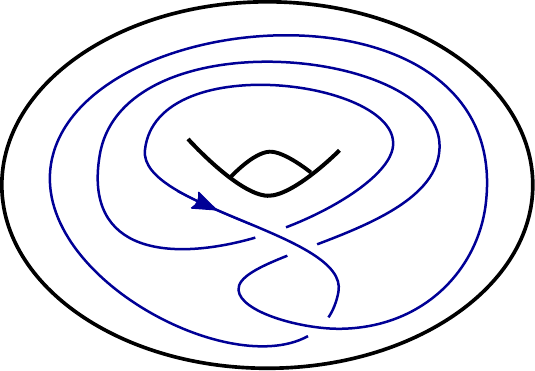} \end{array}
&\begin{array}{c} \includegraphics[height=2.8cm]{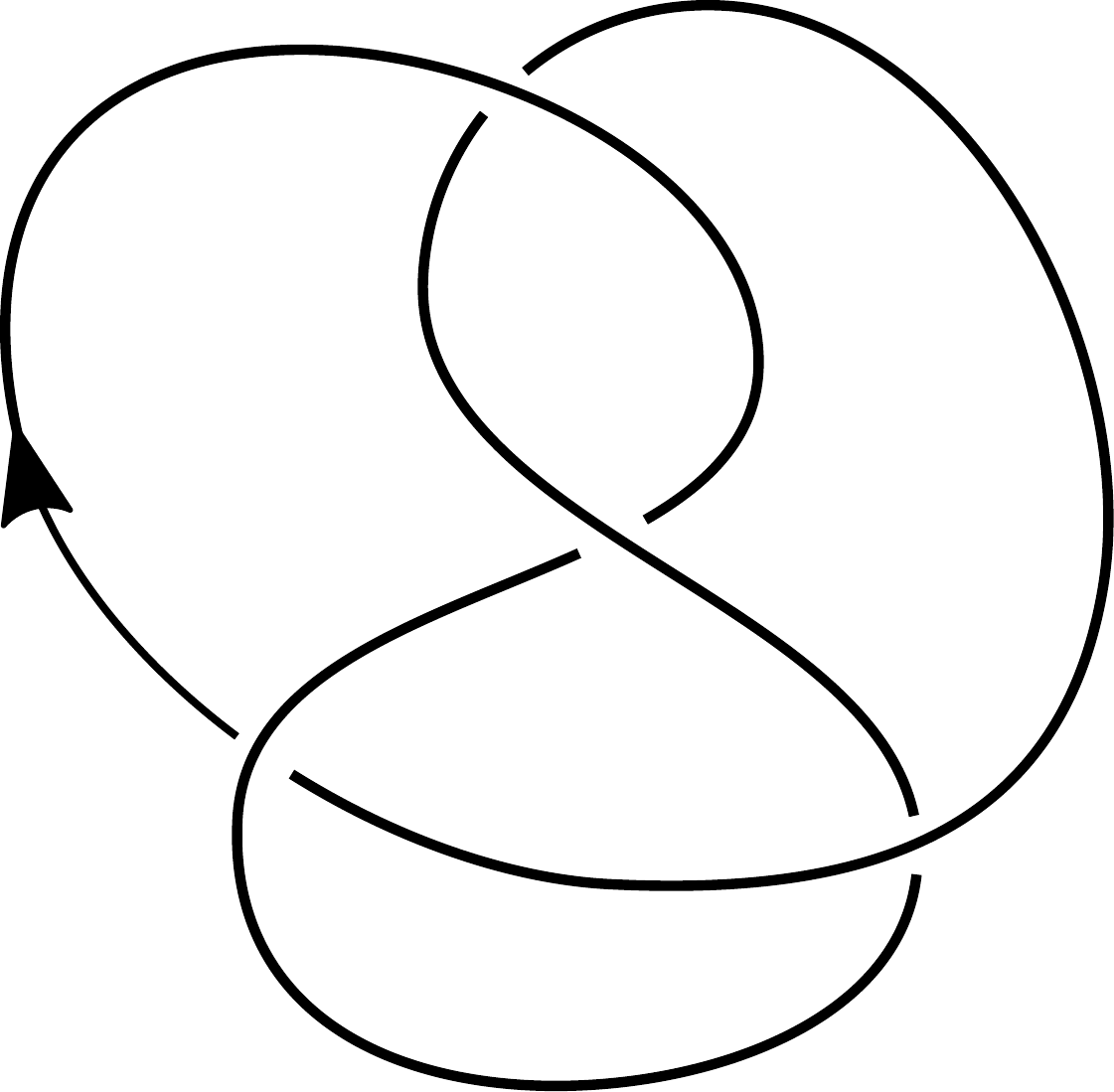} \end{array} &
\begin{array}{c} \includegraphics[height=2.8cm]{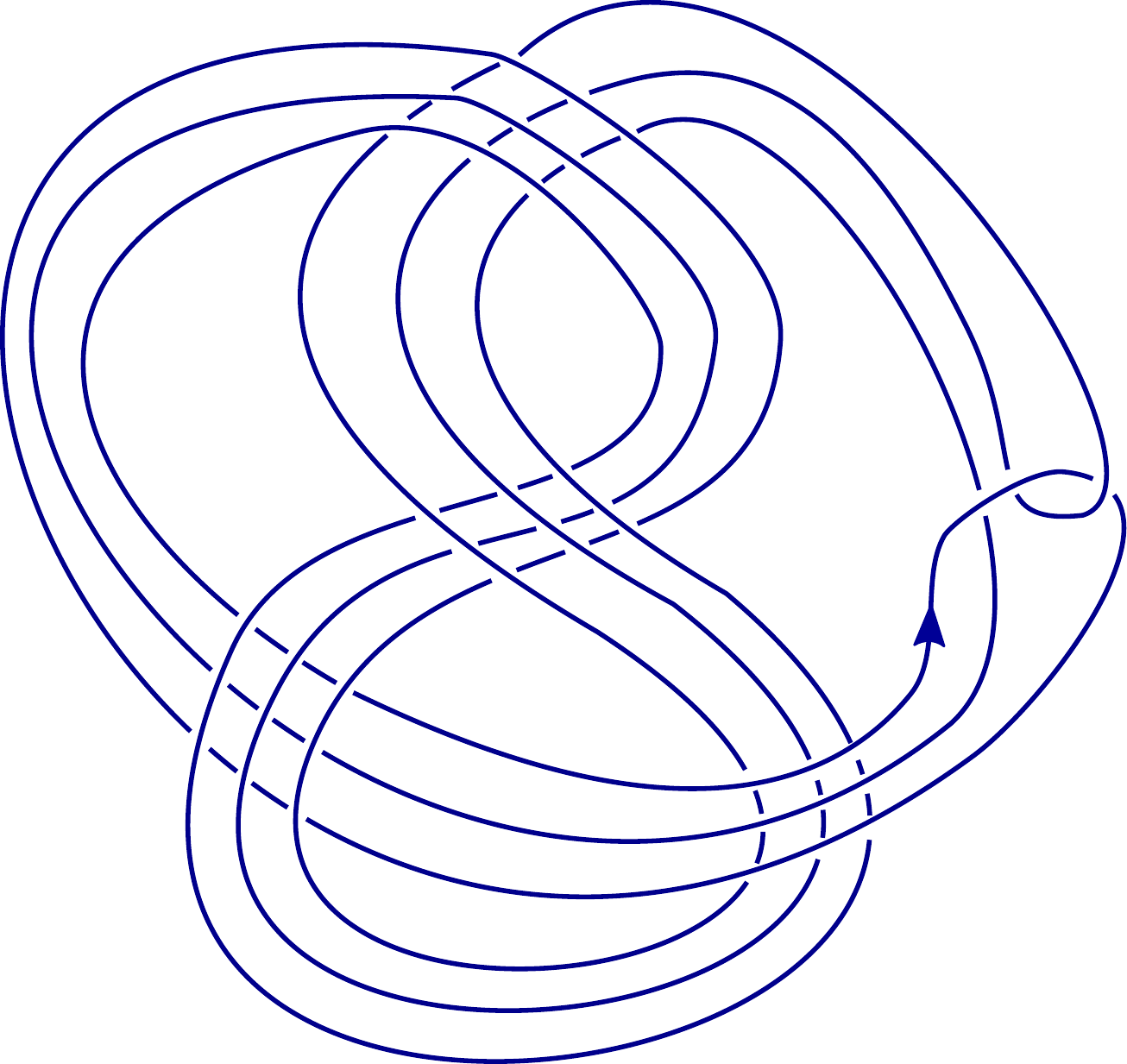} \end{array}
\end{array}$
\caption{A pattern $P$ (left) and companion knot $K$ (center) combine to give the satellite knot $P(K)$ (right).}
\end{figure}

The map $K \mapsto P(K)$ descends to a well-defined function on the collection of (smooth or topological) concordance classes of knots. These satellite maps are essential tools in the modern study of knot concordance and in 3- and 4-manifold topology more generally. To sample just a few results, the satellite construction features prominently in the first evidence for a fractal structure on concordance~\cite{Cochran-Harvey-Leidy:2011-02}; the first examples of  non-smoothly concordant knots with homeomorphic 0-surgeries~\cite{Yas15}; and the first knots in homology spheres which do not bound  PL discs in any contractible 4-manifold~\cite{Lev16}. As a result,  satellite operations have become an object of study in their own right, with recent work in the area focusing on the
 existence of interesting bijective satellite maps~\cite{GM95, MP17}, the behavior of the 4-genera of knots under satelliting~\cite{CH17,Piccirillo-shakegenus, Mil19, FellerMillerPC19}, and on satellite maps with image of infinite rank~\cite{HeddenPinzon}. 

Nonetheless, a fundamental question remains almost entirely open. The collection of concordance classes of knots famously has the structure of an abelian group, with addition induced by connected sum and the inverse operation induced by taking the mirror-reverse of a knot, and it is natural to ask how a satellite operator interacts with this additional structure.

\begin{quest}
When does a pattern induce a homomorphism of the concordance group?
\end{quest}

The following three \emph{standard} patterns evidently induce homomorphisms in both categories:
\[
\begin{array}{ccccc}
\includegraphics[height=2cm]{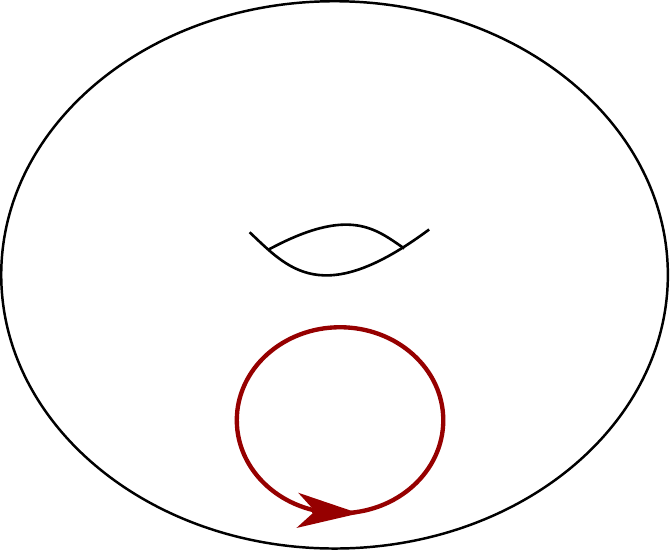} &&
\includegraphics[height=2cm]{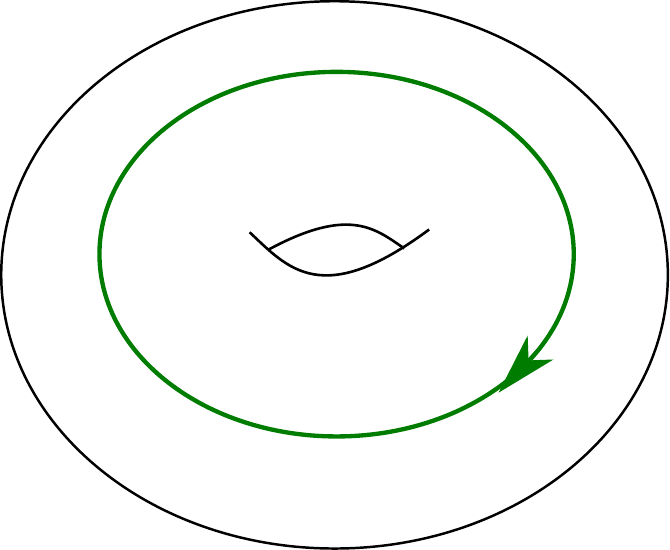} && 
\includegraphics[height=2cm]{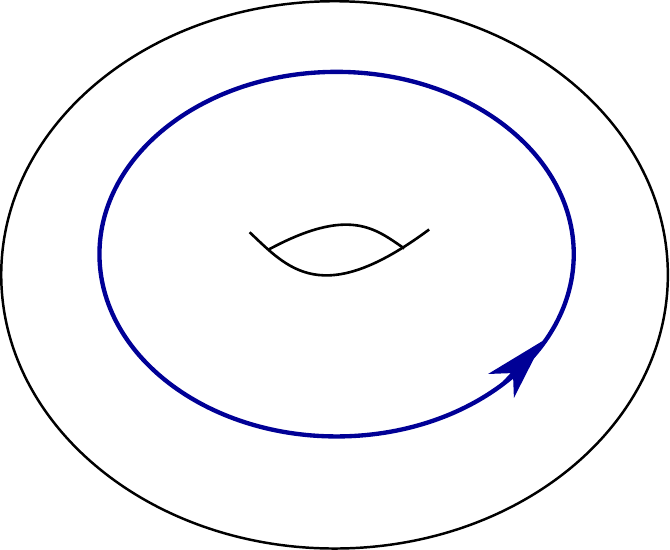} \\
K \mapsto U, && K \mapsto K, && K \mapsto K^r.
\end{array}
\]
 Hedden conjectured that these simple maps are the only homomorphisms  induced by satelliting. 

\begin{conj}[\cite{MPIM16, HeddenPinzon}]\label{conj:nohom}
Let $P$ be a pattern which induces a homomorphism on concordance. 
Then $P$ induces one of the three standard maps on concordance, i.e. the action of $P$ is given by one of
$[K] \mapsto [U]$, $[K] \mapsto [K]$, or $[K] \mapsto [K^{r}].$
\end{conj}

We remark that it remains almost entirely open whether a pattern is determined up to concordance in $(S^1 \times D^2) \times I$ by its action on the concordance group. The exception is the winding number 0 case in the topological category, where we know for example that the Whitehead double pattern is not concordant to the trivial pattern but does induce the zero map on topological concordance. One might therefore hope to strengthen Conjecture~\ref{conj:nohom} to the statement that any pattern inducing a homomorphism must be concordant to a standard pattern, at least in the smooth category. 

We call a pattern $P$ \emph{slice} if $P(U)$ is a slice knot; this is an obvious prerequisite for a pattern to induce a homomorphism.  Perhaps surprisingly, any slice pattern induces a homomorphism of Levine's algebraic concordance group~\cite{Lith84}, 
and in particular `looks like' a homomorphism  from the perspective of classical invariants such as the Tristram-Levine signatures and Alexander polynomial.
Nonetheless, various  modern smooth technologies have been used to show that the simplest non-standard patterns--the Whitehead pattern \cite{Gompf:1986-1}, the Mazur pattern\footnote{This follows immediately from \cite{Lev16}, though we expect it was known to the experts for some time before.}, and the cable $C_{p,1}$ for $p>1$ \cite{Hedden09}--do not  induce  homomorphisms on the smooth concordance group. 

In this paper, we give new evidence for Conjecture~\ref{conj:nohom} in both the smooth and topological categories. In Section~\ref{ss:cg},  we use Casson-Gordon signatures to give the first obstruction to a slice pattern $P$ inducing a homomorphism on the topological concordance group. 

\begin{theoremalph}\label{thm:topobstruction}
Let $P$ be a slice pattern described by an unknot $\eta$ in the complement of $P(U)$ (i.e. $P=P(U) \subset S^3 \smallsetminus \nu(\eta) \cong S^1 \times D^2$). 
 Suppose that there exists some prime $p$ dividing the winding number of $P$ such that  the lifts of $\eta$ to $\Sigma_p(P(U))$ generate the nontrivial group $H_1(\Sigma_p(P(U)))$. 
Then $P$ does not induce a homomorphism on the topological concordance group. 
\end{theoremalph}

In Section~\ref{ss:cgex} we give examples of patterns of every winding number besides $\pm1$ satisfying the conditions of Theorem~\ref{thm:topobstruction}, as well as examples of patterns obstructed from acting as homomorphisms by Proposition~\ref{prop:strongobstruction}, a stronger but harder to state version of Theorem~\ref{thm:topobstruction}.
  We remark that the outstanding case of winding number $\pm1$  seems quite difficult:  remarkably, there are no slice winding number 1 patterns that are known to not induce the identity on topological concordance! Moreover, (non)-existence of such patterns is closely related to longstanding questions such as the topological Akbulut-Kirby and homotopy ribbon conjectures (see~\cite{GM95, MP17}).

The $n$-solvable filtration of~\cite{Cochran-Orr-Teichner:1999-1} plays a central role in the current understanding of  topological concordance; while we omit a precise definition, knots that are $n$-solvable for large $n \in \mathbb{N}$ are `close' to being topologically slice. The Casson-Gordon style techniques of Theorem~\ref{thm:topobstruction} cannot obstruct satellite maps with image deep in the $n$-solvable filtration from inducing homomorphisms. However, we apply results of~\cite{Cochran-Harvey-Leidy:2011-02} to give many examples of patterns mapping arbitrarily deep in the filtration which do not induce homomorphisms.
\begin{theoremalph}\label{thm:nsolvhom}
For any $n \in \mathbb{N}$, there exist infinitely many patterns $P$ with $P(U)$ slice which do not induce homomorphisms on the topological concordance group and which have image contained within  $\mathcal{F}_n$, the collection of $n$-solvable knots.
\end{theoremalph}

We also consider the extent to which non-standard patterns can act like homomorphisms on subsets of the concordance group, proving the following. 
\begin{theoremalph}\label{thm:homlike}
Let $\{K_i\}_{i=1}^m$ be any finite collection of knots.
Then there exists a pattern $P$ that does not induce a homomorphism on the topological concordance group but which has the property that $P(K_i \# K_j)$ is smoothly concordant to $P(K_i) \# P(K_j)$ for all $1 \leq i, j \leq m$. 
\end{theoremalph}

We remark that in particular one can choose $\{K_i\}_{i=1}^m$ to be any finitely generated 2-torsion subgroup of the concordance group, but it remains an interesting open question whether any non-standard pattern acts as a homomorphism when restricted to the subgroup $\{\#^n K\}_{n \in \Z}$ when $K$ represents an infinite order element of the concordance group.

We conclude by switching to the smooth category,  considering the 19 patterns which are presented by two-component links with at most 8 crossings, and almost completely verifying Hedden's conjecture in that setting. 
\begin{theoremalph}\label{thm:smallhoms}
Let $P$ be a pattern presented by a link $P(U) \cup\,  \eta$ with at most 8 crossings. Then $P$ does not induce a homomorphism on  the smooth concordance group, unless perhaps $P(U) \cup\, \eta =$ L8a9, where it is unknown even if $P$ acts by the identity. 
\end{theoremalph}

%

\section*{Acknowledgements}
This paper benefited greatly from thoughtful comments from Lisa Piccirillo and Mark Powell. 
We are also indebted to Wenzhao Chen for informing us of a result of Hartley on amphichiral patterns.
Some of the preparation of this paper was done while the author was supported by NSF grant DMS-1902880, and we gratefully acknowledge that support. 

\section*{Conventions and notation}

All manifolds are assumed to be compact and oriented. We use $\C_s$ to denote the smooth concordance group, $\C_t$ the topological concordance group, and $\C$ when our statements hold in either category. Unless otherwise stated, all patterns are assumed to be slice in the appropriate category. 

Given a pattern $P\colon S^1 \to S^1 \times D^2$, the class of $[P(S^1)]$ equals $m[S^1 \times \{*\}]$ in $H_1(S^1 \times D^2)$ for some $m \in \mathbb{Z}$. We call $m$ the \emph{algebraic winding number} of $P$ and write $w(P)=m$. Given a pattern $P$ with $w(p)=m$, the pattern $P^{rev}$ with reversed orientation has $w(P^{rev})=-m$ and  the property that $P^{rev}(K)$ is isotopic to $P(K)^{rev}$ for all knots $K$. In particular, $P$ induces a homomorphism of $\C$ if and only if $P^{rev}$ does. For convenience, we therefore restrict to patterns of positive winding number.

\section{A Casson-Gordon obstruction}\label{ss:cg}

Given a knot $K$ and prime power $n \in \mathbb{N}$, the first homology group $H:=H_1(\Sigma_n(K))$ of the  $n$-fold cyclic branched cover comes with some additional structure. First, there is a nondegenerate symmetric  form $\lambda \colon H \times H \to \Q/ \Z$ called the \emph{torsion linking form}. A \emph{metabolizer} for $(H, \lambda)$ is a subgroup $M \leq H$ such that $|M|^2= | H|$ and $\lambda|_{M \times M}=0$. There is a $\Z_n$ action on $H$ induced by the action of covering transformations on $\Sigma_p(K)$, and a metabolizer is called \emph{invariant} if this subgroup is set-wise preserved by the $\Z_n$-action. We remark that classical arguments (see also \cite{CG86}) imply that if $K$ is slice then  $H_1(\Sigma_n(K))$ must have an invariant metabolizer.

Our first obstruction to a pattern inducing a homomorphism comes from Casson-Gordon signature invariants.  We will not fully define these, noting only that to any knot $K$, prime power $p$, and map  $\chi \colon  H_1(\Sigma_p(K)) \to \Z_q$ of prime power order, there is an associated \emph{Casson-Gordon signature} $\sigma(K, \chi) \in \Q$ defined in terms of the Witt class of the twisted intersection form of some associated 4-manifold. 
Moreover, Casson-Gordon signatures obstruct topological sliceness as follows.
\begin{thm}[\cite{CG86}]\label{thm:cassongordon}
Suppose $K$ is a topologically slice knot. Then for every prime power $p$ there exists an invariant metabolizer $M \leq H_1(\Sigma_p(K))$  such that if $\chi$ is a prime power order character with $\chi|_M=0$ then 
$\sigma(K, \chi)=0.$
\end{thm}

Theorem~\ref{thm:topobstruction} is a consequence of the following more general obstruction.

\begin{prop}\label{prop:strongobstruction}
Let $P$ be a pattern described by an unknot $\eta$ in the complement of $P(U)$. Let $p$ be a prime dividing the winding number of $P$,  let $H=H_1(\Sig_p(P(U)))$, and denote the first homology classes represented by the $p$ lifts of $\eta$ to $\Sig_p(P(U))$ by $z_1, \dots, z_p \in H$. 

Suppose that for every invariant metabolizer $M \leq H \oplus H \oplus -H$ there exists a character $\chi= (\chi_1 , \chi_2, \chi_3) \colon H \oplus H \oplus -H \to \Z_q$ with $q$ a prime power  and $\chi|_{M}=0$ such that
$
\{ \pm \chi_1(z_i)\}_{i=1}^p$, $\{ \pm \chi_2(z_i)\}_{i=1}^p$, and $\{ \pm \chi_3(z_i)\}_{i=1}^p $
are not identical when considered as sets with multiplicity. Then $P$ does not induce a homomorphism on $\C_t$. 
\end{prop}

\begin{proof}[Proof of Theorem~\ref{thm:topobstruction}, assuming Proposition~\ref{prop:strongobstruction}]
Let $P$ be as in the statement of Theorem~\ref{thm:topobstruction}  and let $H:=H_1(\Sigma_p(P(U)))$. 
Write $|H|=m^2$ for some $m>1$,  let  $q$ be a prime dividing $m$,  and let $k \in \N$ be maximal such that $q^k$ divides $m$.

For any subgroup $S$ of $G:= H \oplus H \oplus -H$, let 
$S_q$ denote the $q$ primary subgroup of $S$ and 
define the $q$-primary annihilator of $S$ to be 
\begin{align*}
A_q(S):= \{ \chi: H \oplus H \oplus -H \to \Z_{q^{6k}} \text{ such that } \chi|_{S}=0\}.
\end{align*}
Note that $|A_q(S)|=|(G/S)_q|= |G_q|/ |S_q|$. 
Now let $M \leq H \oplus H \oplus -H$ be a metabolizer for the linking form and observe that $|M_q|= q^{3k}= |A_{q}(M)|$. 
Let $H^{1}:= H \oplus 0 \oplus 0$, and note that since $|H^{(1)}_q|= q^{2k}$, we have that $|A_q(H^1)|= q^{4k}.$
Since $A_q(H^1)$ and $A_{q}(M)$ are both subgroups of $A_q(0)$, which has order $|G_q|=q^{6k}$, they must have non-zero intersection .

Let $\chi= (\chi_1, \chi_2, \chi_3)$ be a non-zero element of  $A_q(H^1) \cap A_{q}(M)$. 
Since  $\chi|_{H^1}=0$, we have that
$\{ \pm \chi_1(z_i)\}_{i=1}^p=\{0\}_{i=1}^p$. However, since $\chi$ is non-zero and by assumption the lifts of $\eta$ generate $H_1(\Sigma_p(P(U)))$, we must have either $\chi_2(z_i) \neq 0$ or $\chi_3(z_i) \neq 0$ for some $1 \leq i \leq p$. It follows that the sets 
$\{ \pm \chi_1(z_i)\}_{i=1}^p$, $\{ \pm \chi_2(z_i)\}_{i=1}^p$, and $\{ \pm \chi_3(z_i)\}_{i=1}^p $ are not identical, and Proposition~\ref{prop:strongobstruction} applies to show that $P$ does not induce a homomorphism. 
\end{proof}

While Theorem~\ref{thm:topobstruction} is a particularly simple condition to verify, 
Proposition~\ref{prop:strongobstruction} applies in a much broader range of settings,  for instance in Example~\ref{exl:composite}, when $P(U)$ is a composite knot with $H_1(\Sig_p(P(U)))$ a non-cyclic module. 

To prove Proposition~\ref{prop:strongobstruction},  we  will also need  the  following special case of Litherland's formula for the Casson-Gordon signatures of satellite knots.

\begin{prop}[\cite{Lith84}]\label{prop:Litherlandspecialcase}
Let $P$ be a pattern described by an unknot $\eta$ in the complement of $P(U)$. 
Let $p$ be a prime power dividing the winding number of $P$.
Then for any knot $K$, there is a canonical  covering transformation invariant, linking form preserving isomorphism $\alpha \colon H_1(\Sig_p(P(K))) \to H_1(\Sigma_p(P(U)))$ such that given any prime power order character $\chi \colon H_1(\Sigma_p(P(U))) \to \Z_q$, we have
\[\sigma(P(K), \alpha \circ \chi) = \sigma(P(U),\chi)+ \sum_{i=1}^p \sigma_K(\omega_q^{\chi(\eta_i)}),\text{ where } \omega_q= e^{2 \pi i/q}. 
\]
\end{prop}

We remark that at first glance this result seems decidedly unhelpful in showing that $P$ does not induce a homomorphism.  Since $P$ is a slice pattern, many of the $\sigma(P(U), \chi)$ terms must vanish, leaving us with the formula
$\sigma(P(K), \alpha \circ \chi) = \sum_{i=1}^p \sigma_K(\omega_q^{\chi(\eta_i)}).$
Since the Tristram-Levine signatures are additive with respect to connected sum of knots, we see that in many cases
\begin{align*}
\sigma(P(K_1 \# K_2), \alpha \circ \chi) =\sum_{i=1}^p \sigma_{K_1}(\omega_q^{\chi(\eta_i)})+\sum_{i=1}^p \sigma_{K_2}(\omega_q^{\chi(\eta_i)})
=\sigma(P(K_1), \alpha \circ \chi)+\sigma(P(K_2), \alpha \circ \chi)).
\end{align*}
Nonetheless, we are able to prove Proposition~\ref{prop:strongobstruction} as follows.
\begin{proof}[Proof of Proposition~\ref{prop:strongobstruction}]
Let $P$ be as in the statement of the Proposition, and define 
\[C:= \max_{\chi: H_1(\Sig_p(P(U))) \to \Z_q}|\sigma(P(U), \chi) |.
\]

Inductively pick even integers $m_1, \dots,  m_{\lfloor q/2 \rfloor}$ such that 
$m_1>3C$ and $m_j> 3C + p m_{i-1}$ for $j>1$ and  even integers $n_1, \dots, n_{\lfloor q/2 \rfloor}$ such that $n_1> 3C+p m_{\lfloor q/2 \rfloor}$ and $n_j> 3C + p m_{\lfloor q/2 \rfloor} + p n_{j-1}$ for $j>1$. 
 Let $J$ and $K$ be knots such that $\sigma_J(\omega_q^j)= m_j$ and  $\sigma_K(\omega_q^j)=n_j$ for all $1 \leq j \leq \lfloor q/2 \rfloor$. This is possible by the proof of Theorem 1 of Cha-Livingston \cite{Cha-Livingston:2004}, see also \cite{Mil19} for a similar argument.

Now let $M$ be a metabolizer of $H_1(\Sigma_p(L))$. By \cite{Lith84}, there is a canonical, covering transformation invariant, linking form preserving identification  
\[ \beta: H_1(\Sigma_p(L)) \xrightarrow{\cong} H_1( \Sigma_p(P(U))) \oplus  H_1( \Sigma_p(P(U))) \oplus - H_1( \Sigma_p(P(U)).\]
Under this identification, $\beta(M)$ is an invariant metabolizer for $H \oplus H \oplus -H$. Now let $\chi= (\chi_1, \chi_2,\chi_3)$ be a character to $\Z_{q}$ vanishing on $\beta(M)$ such that the sets 
$A_1=\{ \pm \chi_1(z_i)\}_{i=1}^p$, $A_2=\{ \pm \chi_2(z_i)\}_{i=1}^p$, and $A_3=\{ \pm \chi_3(z_i)\}_{i=1}^p $ are not identical.
For $k=1,2,3$ and $1 \leq j \leq \lfloor q/2 \rfloor$, define 
\[ \delta_k(j):= \#\{1 \leq i \leq p : \chi_k(z_i) = \pm j \in \Z_q \}.\]

By Proposition~\ref{prop:Litherlandspecialcase}, we have that 
\begin{align*}
\sigma(L, \alpha \circ \chi)&= 
\sigma(P(J), \chi_1)+
\sigma(P(K), \chi_2)-
\sigma(P(K\#J), \chi_3) \\
&= C_\chi+ \sum_{i=1}^{p} \sigma_J(\omega_q^{\chi_1(z_i)})
+ \sum_{i=1}^{p} \sigma_K(\omega_q^{\chi_2(z_i)})
- \sum_{i=1}^{p} \sigma_{J \#K} (\omega_q^{\chi_3(z_i)}).  \\
&= C_\chi + \sum_{j=1}^{\lfloor q/2 \rfloor} (\delta_1(j)- \delta_3(j)) m_j + \sum_{j=1}^{\lfloor q/2 \rfloor} (\delta_2(j)- \delta_3(j)) n_j,
\end{align*}
where $C_\chi= \sigma(P(U), \chi_1)+
\sigma(P(U), \chi_2)-
\sigma(P(U), \chi_3).$ Note that $|C_\chi| \leq 3C$.

We split our argument into cases as follows. Suppose first that $A_2\neq A_3$. Let $j_0$ be the maximal $j$ with $\delta_2(j) \neq \delta_3(j)$ and assume for convenience that $\delta_2(j_0)>\delta_3(j_0)$.  (The argument for $\delta_3(j_0)>\delta_2(j_0)$ is exactly analogous).
Then
\begin{align*}
\sigma(L, \alpha \circ \chi)
&=  C_\chi + \sum_{j=1}^{\lfloor q/2 \rfloor} (\delta_1(j)- \delta_3(j)) m_j + \sum_{j=1}^{j_0-1} (\delta_2(j)- \delta_3(j)) n_j +(\delta_2(j_0)- \delta_3(j_0)) n_{j_0} \\
&\geq -3C- p m_{\lfloor q/2 \rfloor} - p n_{j_0-1}+n_{j_0}>0, \text{ as desired.}
\end{align*}

Now suppose that $A_2=A_3$ and hence that $A_1 \neq A_3$.  Let $j_0$ be the maximal $j$ with $\delta_1(j) \neq \delta_3(j)$ and as before assume for convenience that $\delta_1(j_0)>\delta_3(j_0)$. 
In this case, we have that 
\begin{align*}
\sigma(L, \alpha \circ \chi)&=  C_\chi + \sum_{j=1}^{j_0-1} (\delta_1(j)- \delta_3(j)) m_j +(\delta_1(j_0)- \delta_3(j_0)) m_{j_0}  \\
& \geq -3C -p m_{j_0-1}+ m_{j_0}>0, \text{ as desired.} \qedhere
\end{align*}
\end{proof}

\section{Examples of non-homomorphism satellite maps}\label{ss:cgex}

An easy way to guarantee that the conditions of Theorem~\ref{thm:topobstruction} are satisfied is to choose a slice knot $P(U)$ whose Alexander module is generated by the class of the winding number 0 curve $\eta$. One can then modify $\eta$ to get a pattern of any winding number.
 More specifically, let $P_n$ be the winding number $n$ pattern of Figure~\ref{fig:pn}, described by an unknot $\eta$ in the complement of $P_n(U)$. Observe that $P_n(U)=6_1$ is slice for all $n \in \mathbb{N}$. 
\begin{figure}[h!]
  \includegraphics[align=c,height=5cm]{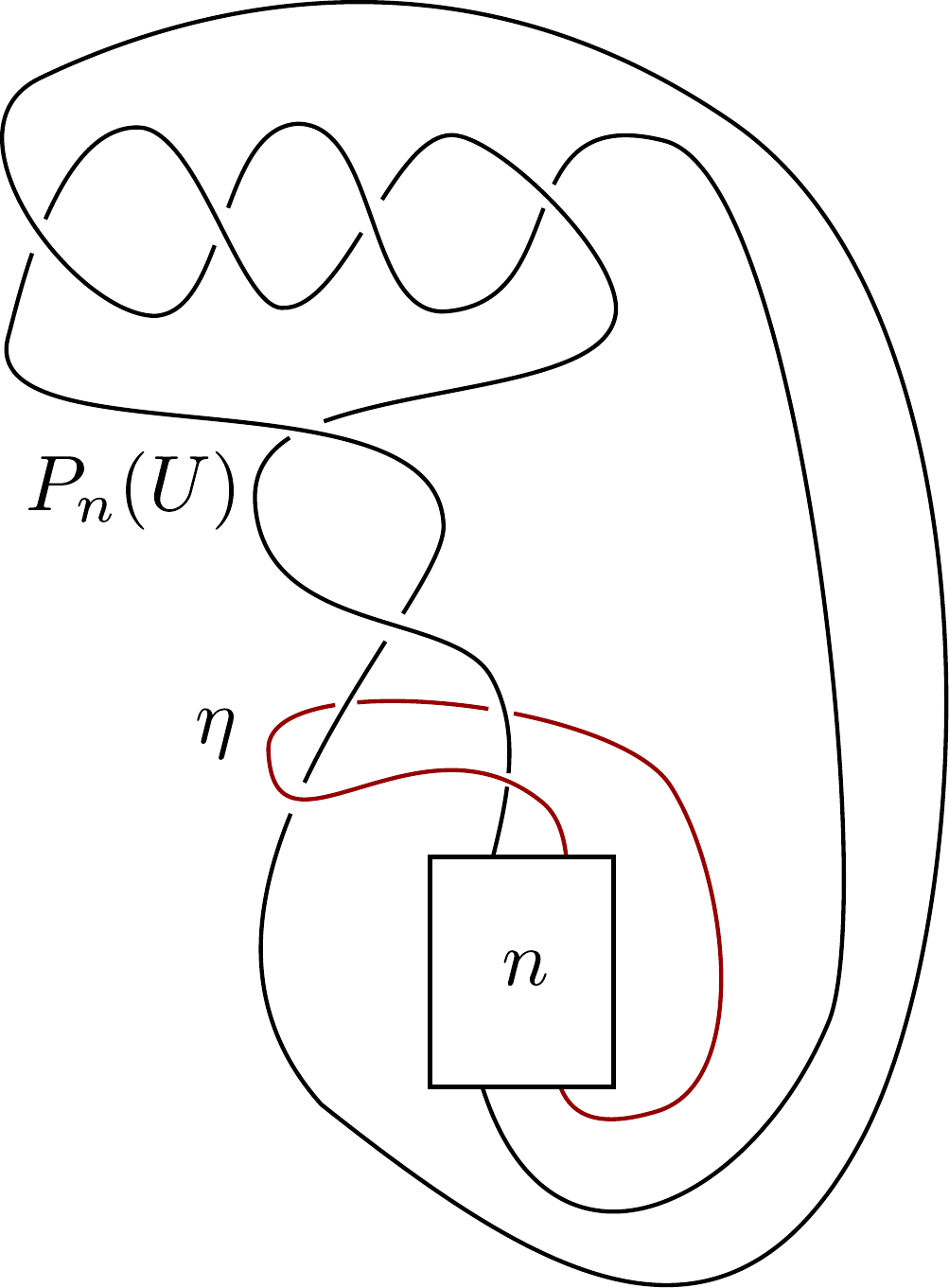}
  \caption{A winding number $n$ pattern $P_n$ in the solid torus $S^3 - \nu(\eta)$.}
  \label{fig:pn}
\end{figure}

\begin{prop}\label{prop: exampleH1}
For any $p$ dividing $n$,  $H_1(\Sigma_p(P_n(U)))$ is a cyclic $\Z[\Z_p]$-module of size $(2^{p}-1)^2$ which is generated by the lifts of $\eta$ to $\Sigma_p(P_n(U))$. 
\end{prop}

\begin{proof}
Let $p$ divide $n$. 
Figure~\ref{fig:surgerypn} gives a surgery description of $P_n(U)$ on the left which is simplified in the center. 
For $p$ dividing $n$, we obtain a surgery diagram for $\Sigma_p(P_n(U))$ with $p+1$ surgery curves, as illustrated for $p=3$ in Figure~\ref{fig:surgerypn}. 
\begin{figure}[h!]  \includegraphics[align=c,height=4cm]{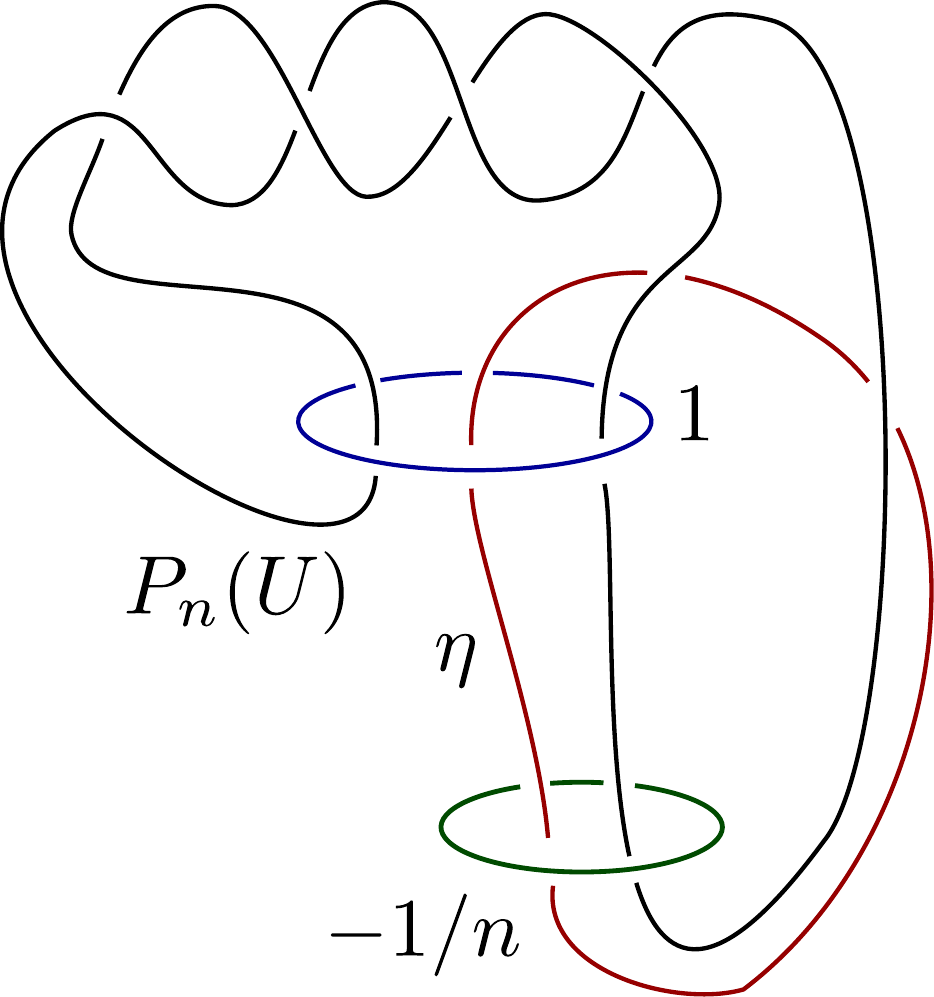}
    \hspace*{.25in}
  \includegraphics[align=c,height=4cm]{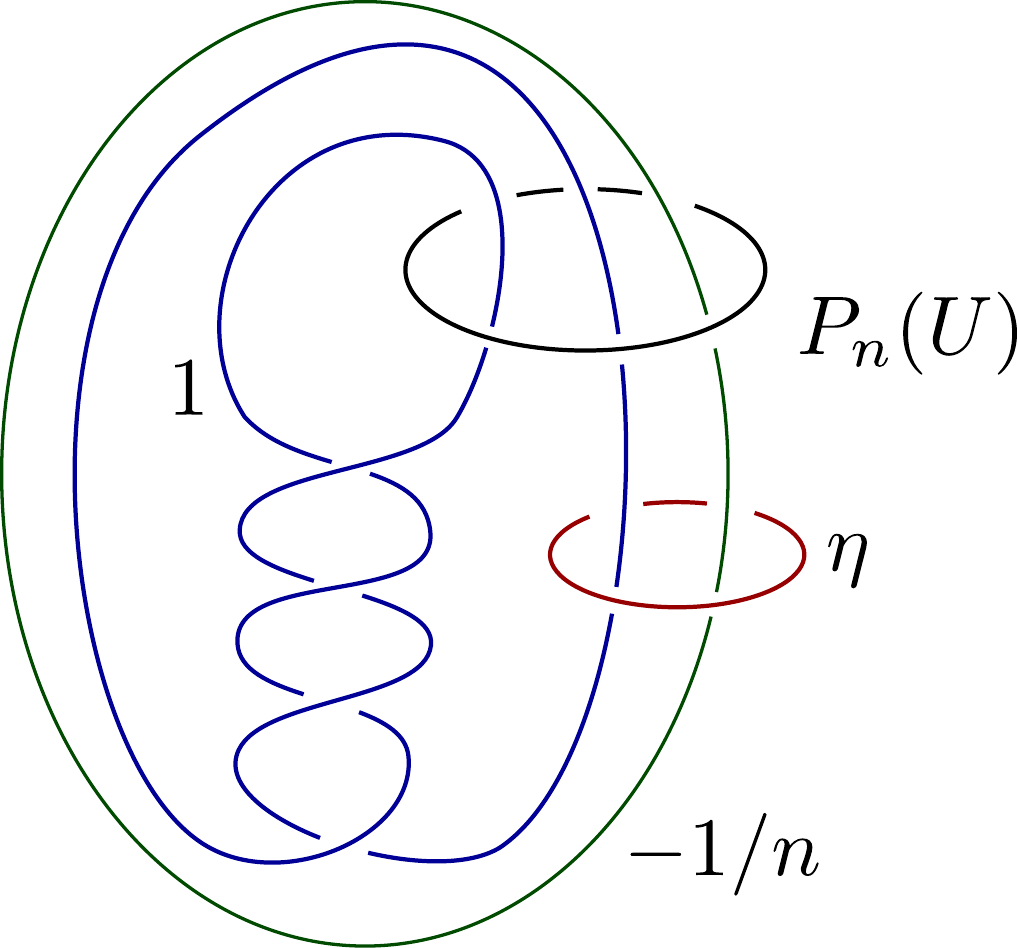}
      \hspace*{.25in}
    \includegraphics[align=c,height=5cm]{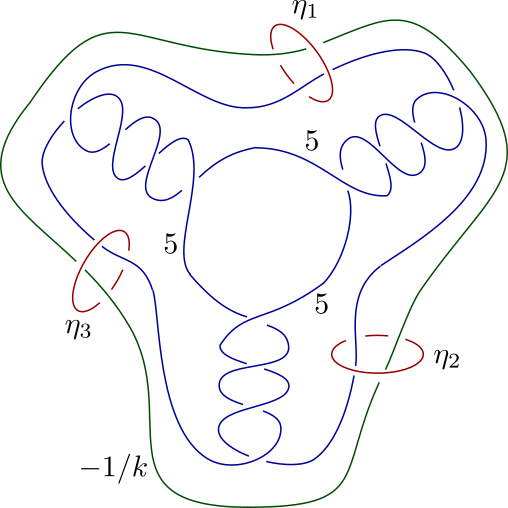}
\caption{A surgery description for $P_n(U)$ (left) is isotoped to an alternate description (center) which lifts to give a surgery diagram for $\Sigma_p(P_n(U))$ for $p$ dividing $n$ (right, depicted for $p=3$ and $n=3k$).}
\label{fig:surgerypn}
\end{figure}

Since the blue $\alpha$ curve has writhe $-4$ and framing $+1$ in the center diagram  and its $p$ lifts  $\alpha_1, \dots, \alpha_p$  will have writhe $0$,  these lifts must be $+5$ framed. Also for $i \neq j$ we have
\[\lk(\alpha_i, \alpha_j)= \left\{ 
\begin{array}{cl}
-2 & \text{ if }  p>2 \text{ and }  j \equiv i  \pm 1 \mod p \\
0 & \text{ if }p>2 \text{ and }   j \not \equiv i  \pm 1 \mod p\\ 
-4 & \text{ if } p=2
\end{array}
\right.
\]
Note that the lifts $\eta_1, \dots, \eta_p$ of the red $\eta$ curve satisfy $\lk(\alpha_i, \eta_j)= \delta_{i,j}$. 
Finally,  the single lift  $\beta$ of the green curve has framing $-1/k$, where $k=n/p$. Note that $\beta$ does not link any of the $\alpha_i$ curves, so we can blow down $\beta$ to obtain a new surgery description without changing the framing of the $\alpha_i$ curves or the linking of the $\alpha_i$ curves with $\eta_j$ curves (though while complicating the diagram significantly!)
After this blow-down, we therefore see that $H_1(\Sigma_p(P_n(U)))$ is generated as a group by the meridians of $\alpha_1, \dots, \alpha_p$, which are cyclically permuted by the covering transformation induced action. So $H_1(\Sigma_p(P_n(U)))$ is a cyclic $\Z[\Z_p]$-module. Also, observe that   each $\eta_i$ is homologous to the meridian of $\alpha_i$ and hence is a $\Z[\Z_p]$-generator for $H_1(\Sigma_p(P_n(U)))$. 

It only remains to show that $|H_1(\Sigma_p(P_n(U)))|=(2^{p}-1)^2$. 
Observe that the group $H_1(\Sigma_p(P_n(U)))$ is presented by the linking matrix of the surgery description. 
For $p=2$, this is $\left[\begin{array}{cc} 5 & -4\\ -4&5 \end{array} \right]$, which is of order $9=(2^2-1)^2$. 
For $p>2$, this is a  $p \times p$ matrix $A(p)$ with 
\[A(p)_{i,j}= \left\{
 \begin{array}{cl}
5 & \text{ if } i=j \\
-2 & \text{ if } i \equiv  j\pm1 \mod p\\
0 & \text{ else}
\end{array}
\right.
.\]
We can check by hand that $\det(A(3))=49= (2^3-1)^2$,  $\det(A(4))= 225= (2^4-1)^2$, and $\det(A(5))=961= (2^5-1)^2$. 
For $p \geq 4$  we define $B(p)$  to be the $p \times p$ matrix with 
\[B(p)_{i,j}= \left\{
 \begin{array}{cl}
5 & \text{ if } i=j \\
-2 & \text{ if } i =  j\pm1 \\
0 & \text{ else}
\end{array}
\right.
.\]
For $p \geq 6$, it then follows from two cofactor expansions that 
\begin{align}\label{eqn:detap}
\det(A(p))= 5 b_{p-1} -8 b_{p-2}-2^{p+1}, \text{ where } b_p:=\det(B(p)),
\end{align}

Observe that $b_4=341$, $b_5=1365$, and that for $p \geq 6$  we can perform two cofactor expansions to show that $b_p= 5b_{p-1}-4b_{p-2}$. Some work with generating functions then shows that for $p \geq 4$ we have $b_p= \frac{1}{3}(2^{2p+2}-1)$. Substituting this expression into Equation~\ref{eqn:detap} gives that
\begin{align*}
\det(A(p))= \frac{5}{3}(2^{2p}-1) -\frac{8}{3}(2^{2p-2}-1)-2^{p+1}= (2^p-1)^2.\quad \quad \quad\qedhere
\end{align*} 
\end{proof}

\begin{cor}\label{cor:pnnothom}
The map induced  by $P_n$ on $\mathcal{C}$ is not a homomorphism for $|n| \neq 1$. 
\end{cor}
\begin{proof}
This follows immediately from Theorem~\ref{thm:topobstruction} in light of Proposition~\ref{prop: exampleH1}. 
\end{proof}

Note that $P_1$ is  geometric winding number 1, and hence acts by connected sum with $P_1(U)= 6_1 \sim U$. However, $P_{-1}$ is not geometric winding number $\pm 1$. While $P_{-1}$  is not topologically concordant to a core of $S^1 \times D^2$ in $(S^1 \times D^2) \times I$ and hence does not obviously act trivially,  to date there  no known ways to show that a slice pattern of winding number $\pm1$ pattern  does not induce a standard map (i.e. identity or reversal) on $\C_t$. 

\begin{prob}\label{prob:pminus1}
Determine whether $P_{-1}$ acts by the identity on $\C_t$ and, if not, determine whether it acts by a homomorphism. 
\end{prob}

\begin{exl}\label{exl:composite}
Let $P$ be the winding number 2 pattern given in Figure~\ref{fig:genpropexample}, and note that $P(U)= 3_1 \#-3_1$.
\begin{figure}[h!]  \includegraphics[align=c,height=3cm]{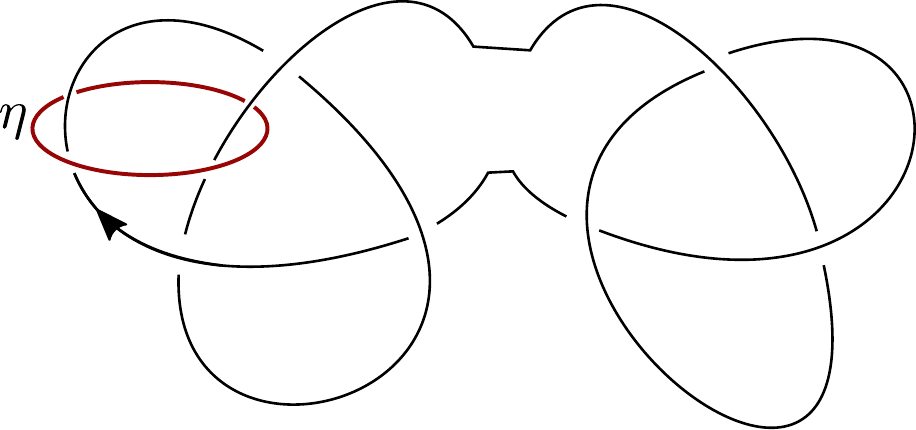}
\caption{A winding number 2 pattern $P$ with $P(U)= 3_1 \# -3_1$.}\label{fig:genpropexample}
\end{figure}
It is straightforward to verify, for example by building a surgery diagram for $\Sigma_2(P(U))$ as in Proposition~\ref{prop: exampleH1}, that $H=H_1(\Sig_2(P(U))) \cong \Z_3  \oplus \Z_3$, with generators $x$ and $y$ such that the linking form $\lambda$ is given by the matrix $\frac{1}{3}\left[ \begin{array}{cc} 1 & 0 \\ 0 & -1 \end{array} \right]$. Moreover, the curve $\eta$ lifts to $\eta_1$ and $\eta_2$ in $\Sig_2(P(U))$, where $[\eta_1]=x$ and $[ \eta_2]=-x$ in $H$. In particular, the lifts of $\eta$ to $\Sig_2(P(U))$ certainly do not generate $H$, and so we cannot apply Theorem~\ref{thm:topobstruction}. Let $x_1, y_1, x_2, y_2, x_3,$ and $y_3$ be the natural generators for $\mathcal{H}=H \oplus H \oplus -H$, where  $\Lambda= \lambda \oplus \lambda \oplus -\lambda$, the linking form on $\mathcal{H}$, is given with respect to our basis by  
\[
\frac{1}{3}\left(\left[ \begin{array}{cc} 1 & 0 \\ 0 & -1 \end{array} \right]
\oplus
\left[ \begin{array}{cc} 1 & 0 \\ 0 & -1 \end{array} \right]
\oplus
\left[ \begin{array}{cc} -1 & 0 \\ 0 & 1 \end{array} \right]
\right).
\]
A straightforward if tedious analysis of the order $27$ subgroups of $\mathcal{H}$ gives us the following list of 48 metabolizers for $\Lambda$, where $\epsilon=(\epsilon_1, \epsilon_2, \epsilon_3) \in (\pm1)^3$.
\begin{align*}
M_1^{\epsilon}&=
 \langle x_1 + \epsilon_1 y_1, x_2 + \epsilon_2 y_2, x_3 + \epsilon_3 y_3 \rangle, \,
 M_2^{\epsilon}=
 \langle x_1 + \epsilon_1 y_1, x_2 + \epsilon_2 x_3, y_2 + \epsilon_3 y_3 \rangle, \\
 M_3^{\epsilon}&=
 \langle x_1 + \epsilon_1 y_2, x_2 + \epsilon_2 y_1, x_3 + \epsilon_3 y_3 \rangle, \,
 M_4^{\epsilon}=
 \langle x_1 + \epsilon_1 y_2, x_2 + \epsilon_2 x_3, y_1 + \epsilon_3 y_3 \rangle, \\
 M_5^{\epsilon}&=
 \langle x_1 + \epsilon_1 x_3, x_2 + \epsilon_2 y_2, y_1+ \epsilon_3 y_3 \rangle, \,
 M_6^{\epsilon}=
 \langle x_1 + \epsilon_1 x_3, x_2 + \epsilon_2 y_1, y_2 + \epsilon_3 y_3 \rangle. 
\end{align*}
We now construct characters to $\Z_3$ vanishing on each $M_j^{\epsilon}$ satisfying the conditions of Proposition~\ref{prop:strongobstruction}. 
\begin{enumerate}
\item If $M= M_1^{\epsilon}$ or $M= M_2^{\epsilon}$, let $\chi=(\chi_1, 0, 0)$, where $\chi_1(x_1)= - \epsilon_1$ and $\chi_1(y_1)=1$. 
Then our collections $\{\chi_j(x_j), \chi_j(-x_j)\}$ for $j=1, 2, 3$  are 
$\{- \epsilon_1, \epsilon_1\}$, $\{0,0\}$ and $\{0,0\}$.
\item If $M=M_3^{\epsilon}$ or $M= M_4^{\epsilon}$, let $\chi= (\chi_1, \chi_2, 0)$ where $\chi_1$ sends $x_1$ to $-\epsilon_1$ and $y_1$ to 0 and $\chi_2$ sends  $x_2$ to 0 and $y_2$ to 1. 
Then our collections $\{\chi_j(x_j), \chi_j(-x_j)\}$ for $j=1, 2, 3$ are  
$\{- \epsilon_1, \epsilon_1\}$, $\{0,0\}$ and $\{0,0\}$.
\item If $M=M_5^\epsilon$ or $M=M_6^{\epsilon}$, let $\chi= (\chi_1, \chi_2, 0)$, where $\chi_1$ sends $x_1$ to 0 and $y_1$ to 1 and $\chi_2$ sends $x_2$ to $-\epsilon_2$ and $y_2$ to 0. 
Then our collections $\{\chi_j(x_j), \chi_j(-x_j)\}$ for $j=1, 2, 3$ are  
$\{0,0\}$, $\{- \epsilon_2, \epsilon_2\}$,  and $\{0,0\}$.
\end{enumerate}

Note that in this example it is important that we only need to consider metabolizers rather than arbitrary subgroups of the appropriate order, since certainly any  character $\chi$ that vanishes on the order 27 subgroup  $\langle x_1, x_2, x_3 \rangle$ will have $\{\chi_i(x_j), \chi_i(-x_j)\}=\{0,0\}$ for all $j=1,2,3$. 
\end{exl}

\begin{rem}
Proposition~\ref{prop:strongobstruction} can never be applied to a pattern $P$ when any of the following conditions hold:
\begin{enumerate}
\item The curve $\eta$ is in the second derived subgroup of $S^3 \smallsetminus P(U)$, and hence lifts to a null-homologous curve in every cyclic branched cover of $P(U)$.
\item The knot $P(U)$ has $\Delta_{P(U)}(t)=1$. 
\item The winding number of $P$ is $\pm 1$.
\end{enumerate}
As discussed before Problem~\ref{prob:pminus1}, the goal of obstructing a pattern from being a homomorphism in case (3) is quite ambitious. However, in the next section we give examples of patterns for each $n \in \N$ which are described by curves lying in the $n$th derived subgroup of $S^3 \smallsetminus P(U)$ and yet which do not induce homomorphisms on $\C_t$. 
\end{rem}

\section{Other topological obstructions}\label{ss:cot}

We consider patterns $R_J$, which are described in Figure~\ref{fig:pattern946}. 
\begin{figure}[h!]
  \includegraphics[height=4cm]{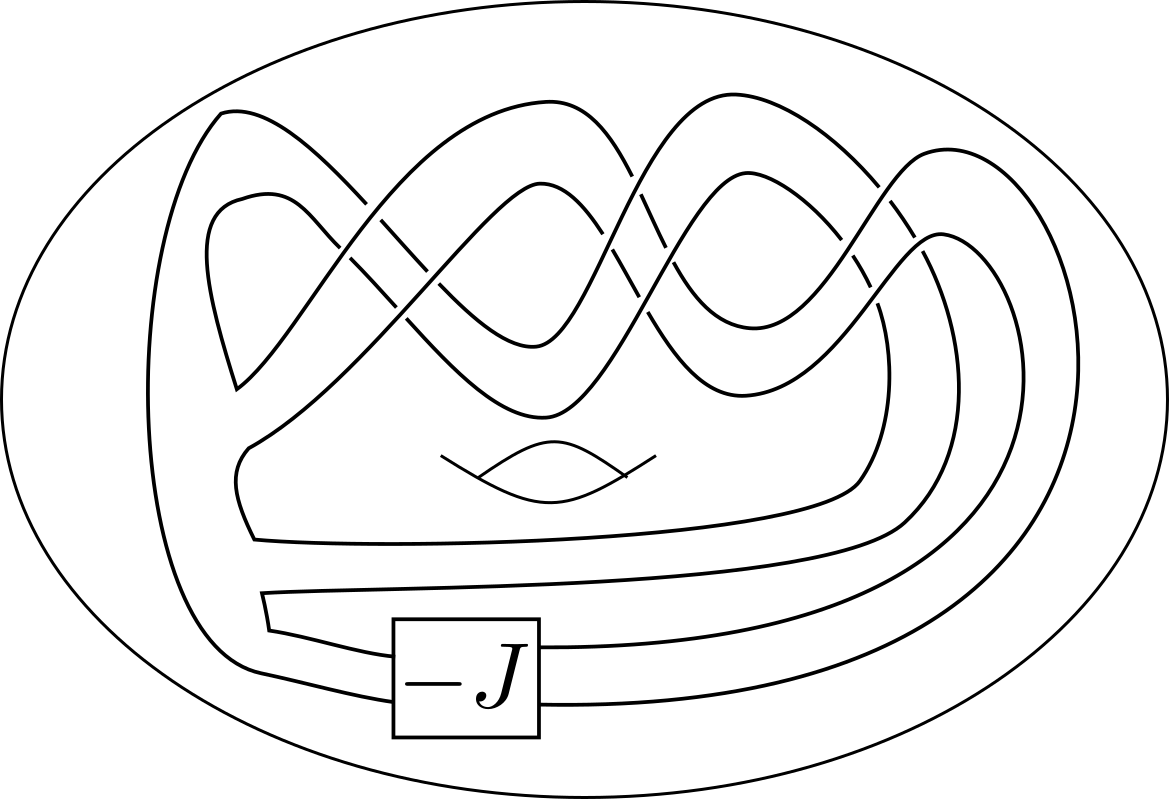} 
  \quad \quad
  \includegraphics[height=4cm]{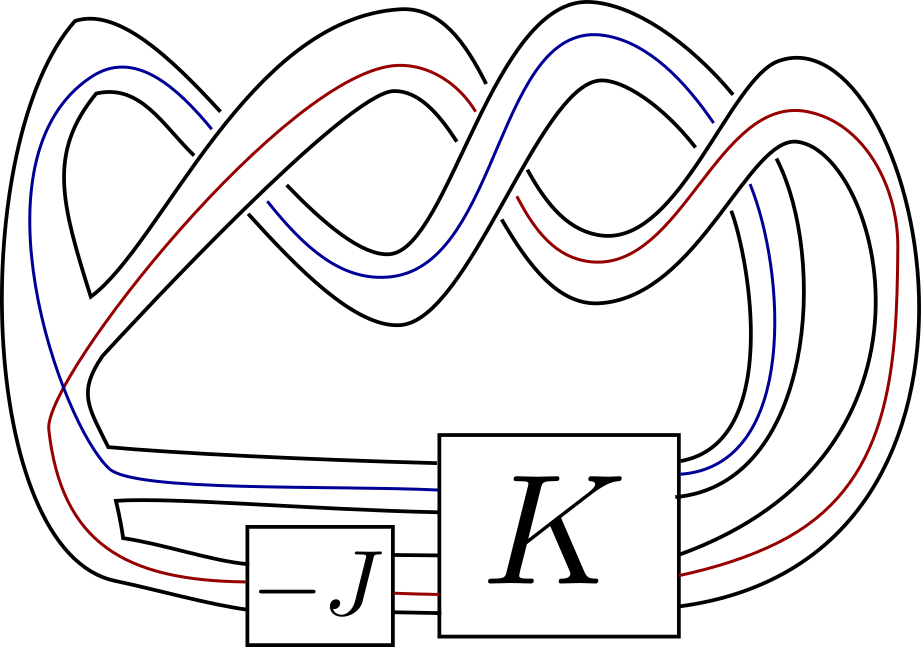}
  \label{fig:pattern946}
  \caption{The pattern $R_J$ (left) and the knot $R_J(K)$ (right), with two curves drawn on its genus 1 Seifert surface.}
\end{figure}
Observe that $R_J(U)$ and $R_J(J)$ can both be seen to be smoothly slice as follows. The knot $R_J(K)$ always has a genus 1 Seifert surface with two 0-framed curves,  shown in blue and red on the right of Figure~\ref{fig:pattern946}. When $K=U$, the blue curve is an unknot and in particular is smoothly slice, so surgery of the pushed-in Seifert surface gives a smooth slice disc in $B^4$ for $R_J(U)$. Similarly, when $K=J$, the red curve has knot type $-J \# J$, which again is smoothly slice and so which can be surgered along to give a slice disc for $R_J(J)$. We now prove that there are many choices of knots $\{J_i\}$ such that arbitrarily many compositions of $R_{J_i}$ maps still do not induce homomorphisms, and as a result give examples of non-homomorphism patterns which map deep into the $n$-solvable filtration of \cite{Cochran-Orr-Teichner:1999-1}.

\begin{rem}
In fact, the pattern $R_J$ never induces a homomorphism. 
We leave the details of this argument to the interested reader, noting that the classes  of $H_1(\Sigma_2(R_J(U)))$ represented by the lifts of $\eta$  are independent of $J$. It is not hard to verify as in Proposition~\ref{prop: exampleH1}  that $H_1(\Sigma_2(R_J(U)))$  is isomorphic to $H:=\Z_3\langle x \rangle \oplus \Z_3 \langle y \rangle$, and has linking form $\lambda$  given with respect to these generators by the matrix $\frac{1}{3} \left[\begin{array}{cc} 0 & -1 \\ -1 & 0 \end{array}\right]$. Moreover, the infection curve $\eta$ lifts to curves representing $x+y$ and $-x-y$ in $H_1(\Sigma_2(R_J(U)))$. An analysis of the metabolizers of $(H, \lambda) \oplus (H, \lambda) \oplus (H,-\lambda)$ as in Example~\ref{exl:composite} now shows that the conditions of Proposition~\ref{prop:strongobstruction} hold. 

For appropriate choices of $\{J_i\}_{i \in \Z}$ (e.g. with increasingly large values of $\sigma_{J_i}(e^{2 \pi i/3})$), it is not hard to prove that we get infinitely many winding number 0 patterns, distinct in their action on $\C_t$, which are obstructed from acting as homomorphisms by Proposition~\ref{prop:strongobstruction}. 
\end{rem}

Theorem~\ref{thm:nsolvhom} follows immediately from the following proposition. 

\begin{prop}\label{prop:nsolvablenonhom}
For any $n \in \N$ and any choices of knots $K_1, \dots, K_n$ there exists a winding number 0 pattern $P$ such that 
\begin{enumerate}
\item $P(K)$ is smoothly slice for each $K=U, K_1, \dots, K_n$. 
\item The image of $P$ is contained in $\mathcal{F}_{n+1}$, the $(n+1)$th level of the solvable filtration. 
\item There exists $C>0$ such that if 
$|\rho_0(K)| = \left| \int_{S^1} \sigma_\omega(K) d \omega \right|>C$
then $P(K) \notin \mathcal{F}_{n+2}$, and in particular is not topologically slice.
\item $P$ does not induce a homomorphism $\C_t \to \C_t$ (or even a homomorphism $\C_t \to C_t/ \mathcal{F}_{n+2}$.)
\end{enumerate}
\end{prop}

We will construct the patterns of Proposition~\ref{prop:nsolvablenonhom} by using compositition of patterns.

\begin{defn}
Given patterns $P \colon S^1 \to S^1 \times D^2$ and $Q: S^1\colon S^1 \times D^2$, we define the \emph{composite pattern} $P \circ Q$ as follows. 
Let $i_Q\colon S^1 \times D^2 \to S^1 \times D^2$ be an embedding of a standard tubular neighborhood of $Q(S^1)$. 
Then $P \circ Q$ is the  pattern
\[P \circ Q \colon S^1 \xrightarrow{P} S^1 \times D^2 \xrightarrow{i_Q} S^1 \times D^2.
\]
We remark that $(P\circ Q)(K)$ is always isotopic to $P(Q(K))$.
\end{defn}

The following special case of a result of Cochran-Harvey-Leidy~\cite{Cochran-Harvey-Leidy:2011-02}  implies that any composition of $R_J$ maps is highly nontrivial on concordance, even modulo terms of the $n$-solvable filtration of \cite{Cochran-Orr-Teichner:1999-1}. It uses the Blanchfield pairing $\Bl$ on the Alexander module of a knot, which takes values in $\Q(t)/ \Z[t^{\pm1}]$. We will not work with the Blanchfield pairing in detail, and therefore omit its definition: it suffices for our purposes to know that if the class of $\eta$ generates the Alexander module then $\Bl(\eta, \eta)$ must be non-zero.

\begin{thm}[\cite{Cochran-Harvey-Leidy:2011-02}]\label{thm:chl}
For each $i=1,...n$, let $R_i$ be a slice knot and $\eta_i$ be an unknotted curve in the complement of $R_i$ such that $\lk(R_i, \eta_i)=0$ and $\Bl(\eta_i, \eta_i) \neq 0$. Let $P_i$ be the pattern obtained by considering $R_i$ in the solid torus $S^3 \smallsetminus \nu(\eta_i)$, and let $P=P_n \circ \dots \circ P_1$.  

For any $k \in \mathbb{N}$, if $K \in \mathcal{F}_k$ then $P(K) \in \mathcal{F}_{n+k}$. 
Also, there exists $C>0$ such that if $K$ is a knot with $|\rho_0(K)|>C$ then $P(K) \notin \mathcal{F}_{n+1}$, and hence is not slice. 
\end{thm}
It is well-known that  the pattern given by $(R, \eta)$ in Figure~\ref{fig:pattern946} satisfies the Blanchfield pairing condition as well as evidently having $\lk(R,\eta)=0$~\cite{Cochran-Harvey-Leidy:2011-02}.

\begin{cor}\label{cor:nonhomviapnk}
Let $P$ be a composition of $n$ patterns $P_i$ described by unknotted curves $\eta_i$ in the complement of $P_i(U)$ such that  $[\eta_i] \in \pi_1(S^3 \smallsetminus \nu(P_i(U)))^{(1)}$ and $\Bl(\eta_i, \eta_i) \neq 0$ for all $i=1, \dots n$.
Suppose that there is some knot $K$ with $\rho_0(K) \neq 0$ such that $P(K)$ is slice. 
Then $P$ does not induce a homomorphism on $\mathcal{C}$, even modulo $\mathcal{F}_{n+1}$. 
\end{cor}
\begin{proof}
Let $C$ be as in the statement of Theorem~\ref{thm:chl}. 
Since $\rho_0(nK)=n\rho_0(K)$ and $|\rho_0(K)|>0$, by taking $n$ sufficiently large we have that $|\rho_0(nK)|>C$ and hence, by Theorem~\ref{thm:chl}, that $P(nK) \neq 0 \in \mathcal{C}/\mathcal{F}_{n+1}$ and in particular is not topologically slice. Therefore $P$ is not a homomorphism, since $P(K)$ is smoothly slice. 
\end{proof}
This corollary implies that any composition of $R_{J_i}$ patterns is not a homomorphism, as long as the `innermost pattern' is based on a knot $J_1$ with $\rho_0(J_1) \neq 0$. Proposition~\ref{prop:nsolvablenonhom} now follows quickly. 
\begin{proof}[Proof of Proposition~\ref{prop:nsolvablenonhom}]
Let $P_0= R_{T_{2,3}}$ and $Q_0= P_0$. 
Inductively for $i=1, \dots, n$, let $P_i= R_{Q_{i-1}(K_i)}$ and let $Q_i=P_i \circ Q_{i-1}$. Let 
$P= Q_n= P_n \circ \dots \circ P_1 \circ P_0.$
 Observe that since $P_i(U) \sim U$ for all $i=0, \dots, n$, we have that $P(U) \sim U$. Also, for each $K_i$ we have that 
 \begin{align*}
 P(K_i)&= (P_n \circ \dots \circ P_{i+1})(P_i(Q_{i-1}(K_i))) \\
& = (P_n \circ \dots \circ P_{i+1})(R_{Q_{i-1}(K_i)}(Q_{i-1}(K_i))) \sim (P_n \circ \dots \circ P_{i+1})(U) \sim U. 
 \end{align*}
 For any knot $K$, the knot $P_0(K)$ is genus 1 and algebraically slice, hence by \cite{DMOP16} is 1-solvable. Therefore, Theorem~\ref{thm:chl} implies that
$P(K)= (P_n \circ \dots \circ P_1)(P_0(K))$ is $(n+1)$-solvable, and we have established condition (2). 
Theorem~\ref{thm:chl} also implies (3) , and since $P(T_{2,3}) \sim U$ and $\rho_0(T_{2,3}) \neq 0$,  Corollary~\ref{cor:nonhomviapnk} implies that (4) holds as well. 
\end{proof}

We can also use  iterated satellite constructions to give the first examples of non-standard patterns which behave like homomorphisms on arbitrary finite sets of knots.
\begin{proof}[Proof of Theorem~\ref{thm:homlike}]
Let  $J_1, \dots, J_n$ be any finite list of knots. Apply Proposition~\ref{prop:nsolvablenonhom} to the collection $\{T_{2,3}\} \cup \{J_i\}_{i=1}^{n} \cup \{J_i \# J_j\}_{i,j=1}^n$ to obtain a pattern $P$ such that for any $i$ and $j$
\begin{align}\label{eqn:homlike}
P(J_i \# J_j) \sim U \sim U \# U \sim P(J_i) \# P(J_j).\end{align}
Corollary~\ref{cor:nonhomviapnk} implies that $P$ does not induce a homomorphism of $\C_t$.
\end{proof}
 We remark that the map induced by $P$ is provably not even a homomorphism on the subgroup generated by the $\{J_i\}_{i=1}^m$, so long as one of the $J_i$ has $\rho_0(J_i) \neq 0$.  Forcing homomorphism-like behavior by sending many knots to the trivial class is not particularly exciting, but the fact that these examples are the first of their type should indicate the wide-open nature of Conjecture~\ref{conj:nohom} and prompt the following question. 

\begin{quest}\label{quest:subhom}
Does there exist a non-standard pattern $P$ which acts by a homomorphism when restricted to some infinite subgroup of $\C$?
\end{quest}
We remark that in Example~\ref{exl:L6a4} we exhibit a non-standard pattern which preserves amphichirality, thereby preserving all known examples of 2-torsion in $\mathcal{C}$. However, it does not seem likely that such a pattern $P$ will have $P(K_1 \# K_2) \sim P(K_1) \# P(K_2)$ for $K_1$ and $K_2$ non-concordant amphichiral knots, and so the above question remains open. 

We close by observing that all the obstructions discussed so far rely only on the homotopy class of $\eta$ in the complement of $P(U)$. It is an interesting open question whether the map $P \colon \C_t \to \C_t$ is determined by this homotopy class (see Problem 3.5 of \cite{AIM19}). If this were so, it would imply that our failure to give  examples of patterns with $P(U)$ unknotted which do not induce homomorphisms on $\C_t$ is unsurprising, at least when $|w(P)| \leq 1$: when $P(U)=U$ the linking number of $P$ with $\eta$ determines the homotopy class.  

On the other hand, we find it surprising that one cannot obstruct the $(p,1)$ cable maps from inducing homomorphisms of $\C_t$, and state that as a worthwhile problem.
\begin{prob}
Determine whether the cable $C_{p,1}$ induces a homomorphism of $\C_t$ for $p>1$. 
\end{prob}

\section{Small patterns acting on $\C_{s}$}\label{section:small}

We conclude by considering patterns of small crossing number. Since the first nontrivial slice knot has 6 crossings (either the Stevedore's knot $6_1$ or the square knot $T_{2,3} \# - T_{2,3}$), it is perhaps unsurprising that all of the patterns we consider have $P(U)=U$. As a consequence, none of our topological obstructions apply and so in this section we work only in the smooth category, in particular letting $\sim$ denote the equivalence relation of smooth concordance.

We rely heavily on the $\tau$-invariant of Heegaard-Floer homology \cite{Ozsvath-Szabo:2003-1}, which vanishes on smoothly slice knots. We begin by reviewing the formula for the $\tau$-invariant of $P(K)$ for certain prototypical patterns $P$ of each winding number and arbitrary knots $K$. 

\begin{thm}[\cite{Hom2014-Tau-of-Cables}]
Let $p>1$. Then 
\begin{align*}
\tau(C_{p,1}(K))&=
\left\{
\begin{array}{cc}
p \tau(K), & \epsilon(K)\in\{0,1\}\\
p \tau(K) + (p-1), & \epsilon(K)=-1\\
\end{array}
\right. 
\end{align*}
\end{thm}
Note that this gives an easy proof that $C_{p,1}$ does not induce a homomorphism on $\C_s$ as follows:
Let $K$ be any knot with $\epsilon(K)=+1$ (e.g. $K=T_{2,3}$). Then $\epsilon(-K)=-1$ and $C_{p,1}(K) \# C_{p,1}(-K)$ is not slice, since
\[ \tau(C_{p,1}(K) \# C_{p,1}(-K)) = \tau(C_{p,1}(K))+ \tau(C_{p,1}(-K))= 
p\tau(K) + p \tau(-K) + p-1=p-1 \neq 0.
\]

We also have the following similar results, which give analogous simple proofs that the Mazur and Whitehead patterns do not induce homomorphisms on $\C_s$. 

\begin{thm}[\cite{Lev16}]
Let $\M$ denote the positive Mazur pattern. Then 
\begin{align*}
\tau(\M(K))=
\left\{
\begin{array}{cc}
\tau(K), &  \tau(K) \leq 0 \text{ and } \epsilon(K) \in \{0,1\}\\
\tau(K) + 1, & \text{ else}
\end{array}
\right. .
\end{align*}
\end{thm}

\begin{thm}[\cite{Hedden:2007-1}]
Let $\Wh$ denote the positive Whitehead double pattern. Then 
\begin{align*}
\tau(\Wh(K))=
\left\{
\begin{array}{cc}
0, &  \tau(K) \leq 0 \\
 1, & \tau(K)>0
 \end{array}
\right. .
\end{align*}
\end{thm}

Somewhat surprisingly, even though these `prototypical' patterns do not induce homomorphisms, we can still use these formulae for $\tau$ to prove the following.

\begin{prop}\label{prop:tauhom}
Let $P$ be a pattern that induces a homomorphism on $\mathcal{C}_s$.  
Then for any knot $K$, \[\tau(P(K))= |w(P)| \tau(K).\]
\end{prop}
\begin{proof}
Let $Q=C_{p,1}$ (if $p:=w(P) \geq 1$) and $\Wh$ (if $p=0$.). We give the argument for $Q=C_{p,1}$, but an exactly analogous one works for $\Wh$.
 
Since $P$ and $Q$ have the same winding number, by \cite{CH17} there exists a constant $c>0$ such that
for all $K$, 
\[|\tau(Q(K) \# -P(K))| \leq g_4(Q(K) \#-P(K)) \leq c.\]
Suppose now that $\epsilon(K)=-1$. It follows from the basic properties of $\epsilon$ that $\epsilon(nK)=-1$ for any $n>0$. 
We therefore have the following, using in the first equality that $P$ is a homomorphism:
\begin{align*}
c &\geq |\tau(Q(nK) \# - P(nK)) | = |p \tau(nK)+(p-1) - \tau(nP(K))| 
=|n[p \tau(K)-  \tau(P(K))] + p-1|.
\end{align*}
Letting $n \to \infty$, we see that we must have $\tau(P(K))= p\tau(K)$. 
The argument for $K$ with $\epsilon(K)=+1$ or $\epsilon(K)=0$ is analogous. 
\end{proof}
We remark that it is perhaps an interesting problem to show the same result for Rasmussen's $s$-invariant, which shares many but not all formal properties with $\tau$. Inspection of the proof shows that it would suffice to show that  for each $p$ there exists a constant $C(p)$ such that $|s(C_{p,1}(K)) - p s(K)| \leq C(p)$ for all knots $K$.  The work of Van Cott~\cite{VanCott10} gives bounds for $s(C_{p,q}(K))$ as $q \to \infty$, which seem ill-suited to the case of interest. Of course, if one believes Conjecture~\ref{conj:nohom}, then the $s$-invariant analogue of Proposition~\ref{prop:tauhom} would be trivially true, independently of the behavior of $s$ under cabling. 
\\

It will be useful for us to have a much weaker notion of preserving group structure. 
\begin{defn}
An \emph{pseudo-homomorphism} of a group $G$ is a map $\phi \colon G \to G$ such that $\phi(e_G)=e_G$ and $\phi(g^{-1})=\phi(g)^{-1}$ for all $g \in G$.
\end{defn}
We can rephrase this in our context in a somewhat surprising way. For any pattern $P$ and knot $K$, we have that $-P(K)$ is isotopic to $(-P)(-K)$. It follows that  $P$ induces a pseudo-homomorphism on $\C$  if and only if $P(U) \sim U$ and
$(-P)(K)= -P(-K) \sim P(K)$ for all $K$. 
 
\begin{cor}\label{cor:pseudohom}
Let $P$ be a winding number $p$ pattern. 
Suppose that 
\begin{enumerate}
\item[($p>1$)]
$P$ can be changed to $C_{p,1}$ with any number of crossing changes  $(+)$ to $(-)$  and strictly fewer than $\frac{p-1}{2}$ crossing changes $(-)$ to $(+)$.
\item[($p=1$)]
$P$ can be changed to $\M$ with any number of crossing changes, all $(+)$ to $(-)$. 
\item[($p=0$)]
$P$ can be changed to $\Wh$ with any number of crossing changes, all $(+)$ to $(-)$.
\end{enumerate}
Then  $P$ does not induce a pseudo-homomorphism on $\C_s$. 
\end{cor}
\begin{proof}
Any of the above conditions implies that $\tau(P(T_{2,3}) \# P(-T_{2,3})) >0$, since if $K_+$ and $K_-$ differ by changing a single crossing from $(+)$ to $(-)$ then 
\[ \tau(K_-) \leq \tau(K_+) \leq \tau(K_-)+1. \qedhere\]
\end{proof}
We remark that Proposition~\ref{prop:tauhom} along with the crossing change inequality for $\tau$ implies that if $P$ can be changed to $C_{p,1}$ with any number of crossing changes  $(+)$ to $(-)$  and strictly fewer than $(p-1)$ crossing changes $(-)$ to $(+)$, then $P$ is not a homomorphism. We will see in Example~\ref{exl:L6a4} that this weaker assumption does not obstruct $P$ from inducing a pseudo-homomorphism. \\

The reference tables of LinkInfo~\cite{LinkInfo} give 30  prime 2-component links which have diagrams with no more than 8 crossings, considered independently of orientation and without considering mirror images. 
By picking an unknotted component $\eta$ of such a link, we obtain a pattern $P$ in the solid torus $S^3 \smallsetminus \nu(\eta)$. 
There are 19 choices which define a slice pattern, coming from 18 different links. The link L8a1 is asymmetric, as detected by the multivariable Alexander polynomial, and hence defines two patterns which we call L8a1a and L8a1b. 

 Two of these patterns are standard, as depicted in Figure~\ref{fig:concordanttocore}. 
\begin{figure}[h!]
  \includegraphics[height=2.5cm]{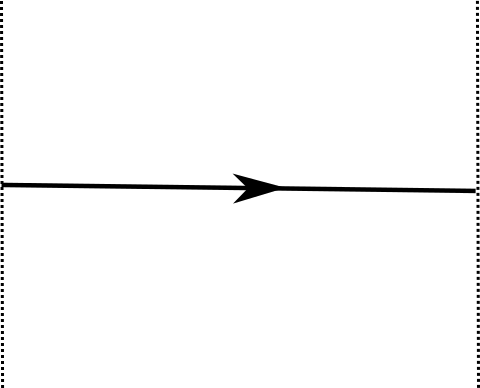} \quad \quad
  \includegraphics[height=2.5cm]{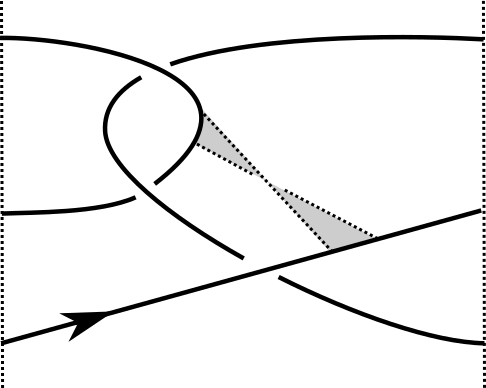}
  \caption{Small patterns concordant to a core: L2a1 (left) and L7a5 (right).}
  \label{fig:concordanttocore}
\end{figure}
(Note that in this section, for efficiency's sake we depict patterns as living in $D^2 \times I$. An untwisted identification of $D^2 \times  \{0\}$ with $D^2 \times \{1\}$ gives the pattern in the solid torus.) 
 
 Corollary~\ref{cor:pseudohom} immediately implies that 12 of the remaining 17 do not induce pseudo-homomorphisms: the necessary crossing changes are illustrated in Figure~\ref{fig:manyexl}.
\begin{figure}[h!]
$
\begin{array}{cccc}
  \includegraphics[height=2.5cm]{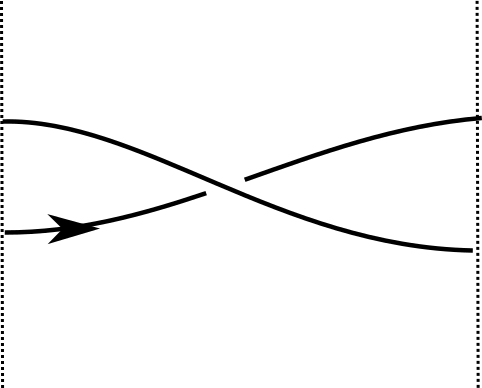} &  \includegraphics[height=2.5cm]{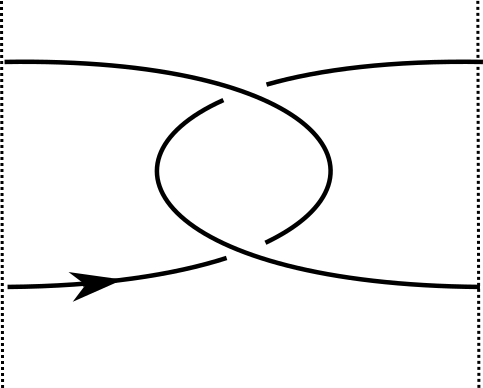} &  \includegraphics[height=2.5cm]{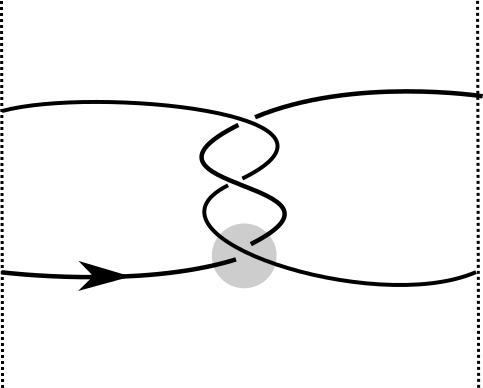}&   \includegraphics[height=2.5cm]{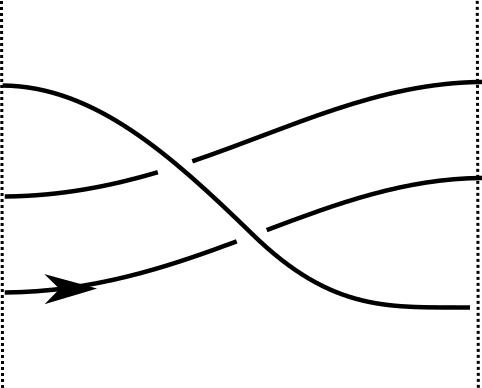}   \\
  \text{L4a1}=C_{2,1}&\text{L5a1=}\Wh & \text{L6a1}  & \text{L6a3}=C_{3,1}\\
    \includegraphics[height=2.5cm]{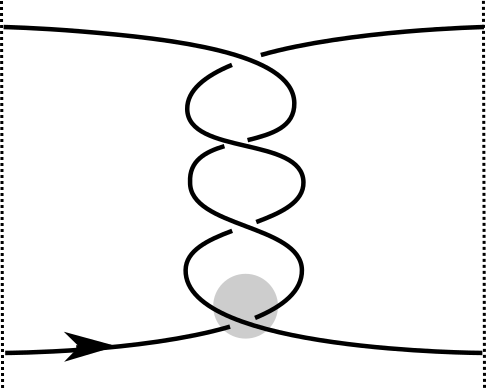}  &   \includegraphics[height=2.5cm]{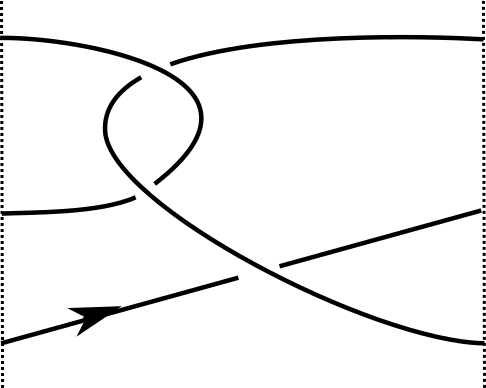} &
  \includegraphics[height=2.5cm]{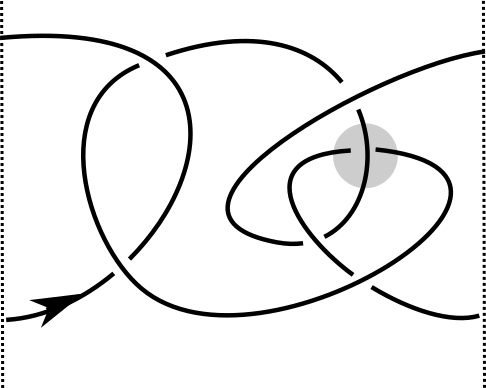}  &   \includegraphics[height=2.5cm]{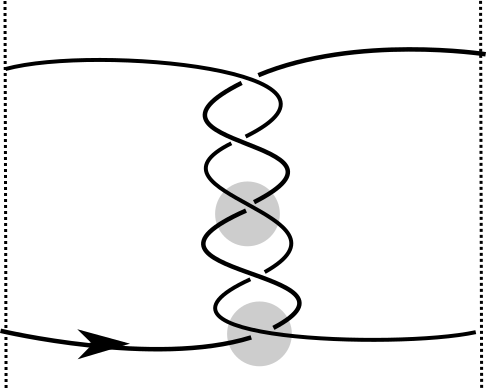}  \\
    \text{L7a4} & \text{L7a6}=\M & \text{L8a1a} & \text{L8a6}\\
      \includegraphics[height=2.5cm]{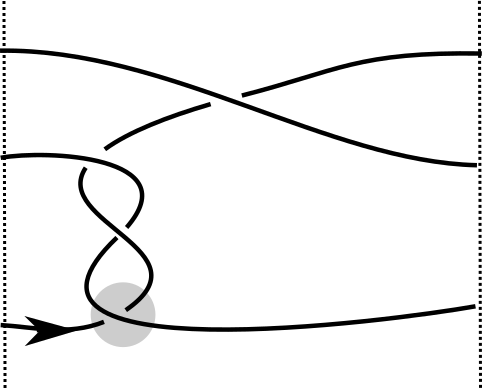} &   \includegraphics[height=2.5cm]{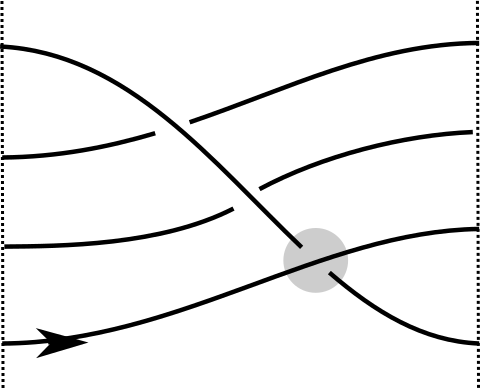} &   \includegraphics[height=2.5cm]{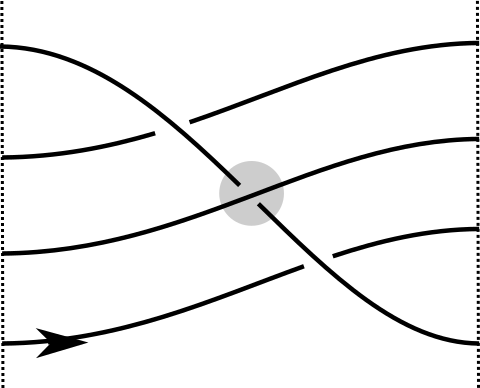} & \includegraphics[height=2.5cm]{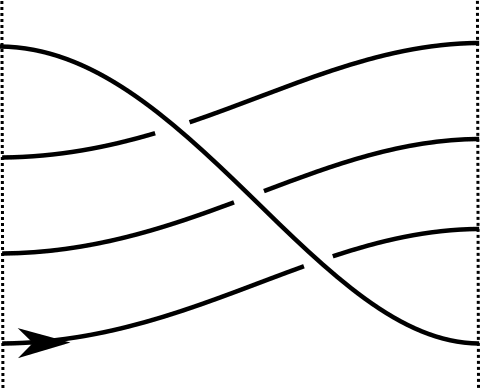}\\
      \text{L8a11} & \text{L8a12} & \text{L8a13} & \text{L8a14}=C_{4,1}
    \end{array} 
$
\caption{Small patterns satisfying the hypotheses of Corollary~\ref{cor:pseudohom}.}\label{fig:manyexl}
\end{figure}

This leaves us with 5 patterns to consider individually. We now give specific arguments to show that L8a1b, L8a8, and L8a10, depicted in  Figure~\ref{fig:special}, do not induce pseudo-homomorphisms. 
\begin{figure}[h!]
  \includegraphics[height=2.5cm]{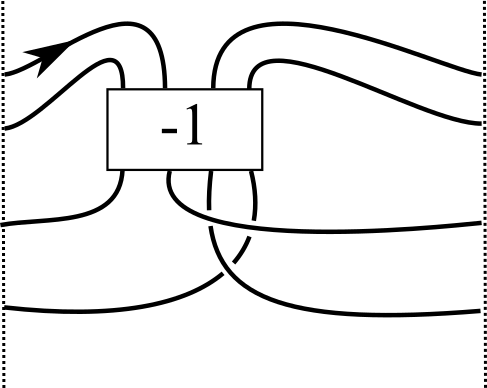} \quad \quad
  \includegraphics[height=2.5cm]{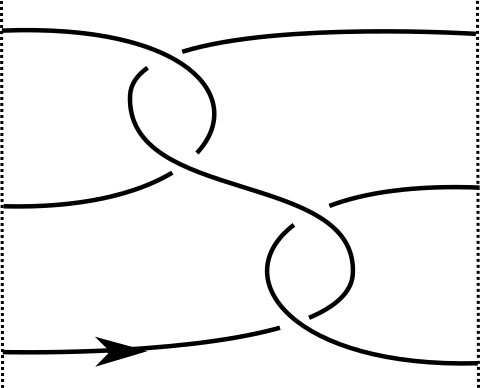}  \quad \quad
  \includegraphics[height=2.5cm]{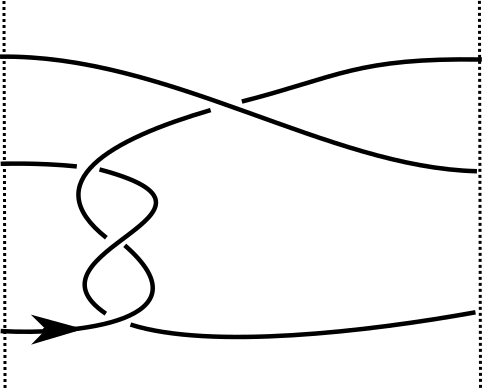}
  \caption{More small patterns: L8a1b (left), L8a8 (center), and L8a10 (right).}
  \label{fig:special}
\end{figure}

\begin{exl}[L8a1b does not induce a pseudo-homomorphism]
The crossing change inequality for $\tau$ generalizes to give the following result. (Note that one obtains a $(+)$ to $(-)$ crossing change by doing a $+1$ twist along a small linking number 0, geometric linking number 2 curve.)

\begin{prop}[\cite{OS03abs}]
Let $K$ be a knot in $S^3$ and $\eta$ be an unknot in the complement of $K$ such that $\lk(K, \eta)=0$. 
Let $K^+$ be the knot obtained from $K$ by doing a $+1$-twist along $\eta$. Then 
$\tau(K^+) \leq \tau(K) \leq \tau(K^+)+1.
$
\end{prop}
Now, observe there is a $+1$-twist along a linking number 0 unknot that takes the pattern L8a1b to the positive Whitehead pattern, as illustrated in Figure~\ref{fig:L8a1b}.
\begin{figure}[h!]
  \includegraphics[height=2.5cm]{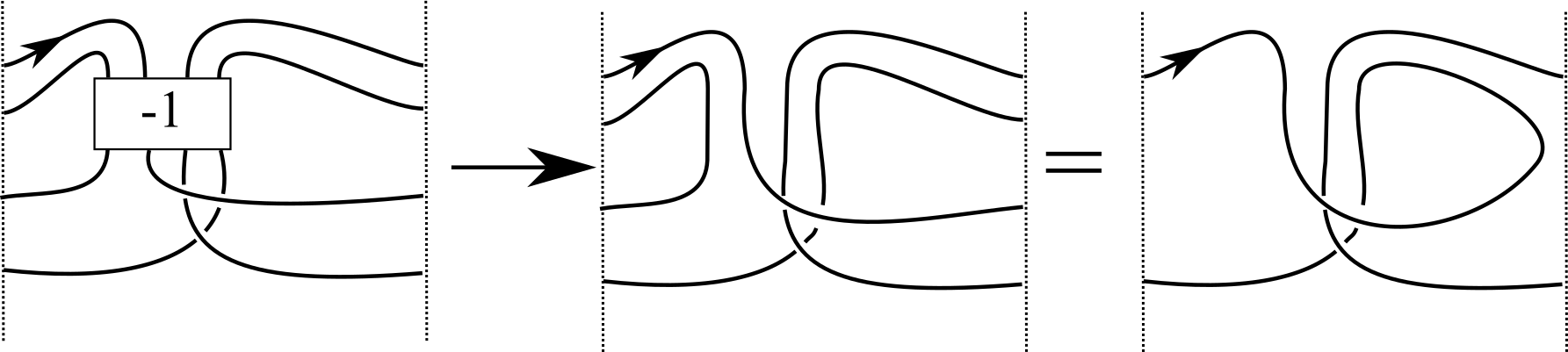}
  \caption{Twisting L8a1b (left) to $\Wh$ (right).}
  \label{fig:L8a1b}
\end{figure}
It follows that for any $K$,  $\tau(\text{L8a1b(K)}) \geq \tau(\Wh(K))$, and so the arguments of the proof of Corollary~\ref{cor:pseudohom} apply to show that L8a1b does not induce a pseudo-homomorphism
\end{exl}

\begin{exl}[L8a8 does not induce a pseudo-homomorphism]\label{exl:8a8}
Since a single (+) to (-) crossing change takes $P=\text{L8a8}$ to a core, we immediately have that $\tau(P(K)) \geq \tau(K)$ for all knots $K$. 
We will now show that this is not always equality, and therefore that for some knot $J$
\[ \tau(P(J) \# P(-J))= \tau(P(J)) + \tau(P(-J)) \geq \tau(P(J)) + \tau(-J)> \tau(J)+ \tau(-J)=0.\]

Observe that L8a8 has a Legendrian diagram (on the left of Figure~\ref{fig:L8a8}) with Thurston-Bennequin number and rotation number equal to 
\begin{align*}
tb(\mathcal{P})&= \text{writhe}(\mathcal{P})- \#(\text{right cusps})=4-2=2.\\
rot(\mathcal{P})&=\frac{ \#(\text{down cusps})- \#(\text{up cusps})}{2}=\frac{2-2}{2}=0.
\end{align*}

There is also a Legendrian diagram  for $J= T_{2,3}$ with 
\begin{align*}
tb(\mathcal{J})&= \text{writhe}(\mathcal{J})- \#(\text{right cusps})=3-3=0.\\
rot(\mathcal{J})&=\frac{ \#(\text{down cusps})- \#(\text{up cusps})}{2}=\frac{4-2}{2}=1.
\end{align*}

\begin{figure}[h!]
  \includegraphics[height=2.5cm]{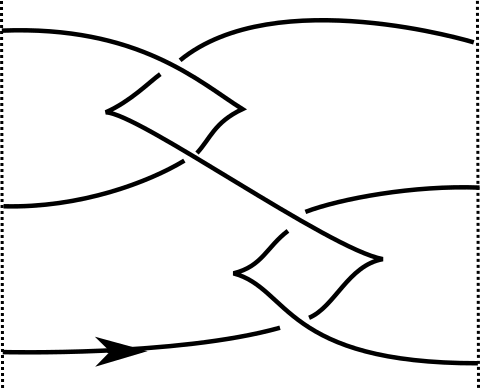} \qquad \qquad
    \includegraphics[height=2.5cm]{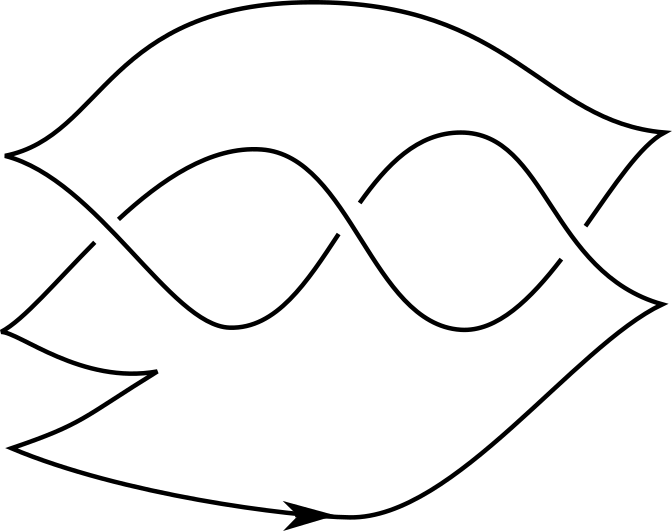}
  \caption{Legendrian diagrams of L8a8 (left) and a right-handed trefoil $J$ (right).}
  \label{fig:L8a8}
\end{figure}

As shown by Ng-Traynor~\cite{NgTraynor}, since $tb(\mathcal{J})=0$, we obtain a Legendrian diagram $\mathcal{P(J)}$ for $P(J)$ with 
\begin{align*}
tb(\mathcal{P(J)})&= w(P)^2 tb(\mathcal{J}) + tb(\mathcal{P})
= (1)^2 \cdot 0 + 2=2\\
rot(\mathcal{P(J)})&=w(P) rot(\mathcal{J})+ rot(\mathcal{P})
=1 \cdot 1+0=1.
\end{align*}
We now apply the following result of Plamenevskaya. 
\begin{thm}[\cite{Plam-tbtau}]\label{thm:tbtau}
Let $\mathcal{K}$ be a Legendrian representative of $K$. 
Then
\[tb(\mathcal{K})+|rot(\mathcal{K}| \leq 2 \tau(K)-1. 
\]
\end{thm}
So for $J=T_{2,3}$ we have 
\[\tau(P(J)) \geq (1/2)(tb(\mathcal{P(J)})+ rot(\mathcal{P(J)})+1)= (1/2)(2+1+1)=2 >1= \tau(J). 
\]
\end{exl}

\begin{exl}[L8a10 does not induce a pseudo-homomorphism]
Let $K=T_{2,3}$. We will use the alternate definition of pseudo-homomorphism and show that L8a10$(K)$ and $(-\text{L8a10})(K)$ are not concordant. 
Since a single $(-)$ to $(+)$ crossing change takes L8a10 to L6a4, we have that 
\[ \tau(\text{L8a10}(T_{2,3})) \leq \tau(\text{L6a4}(T_{2,3})) \leq 3 \tau(T_{2,3})-1=2.\]
\begin{figure}[h!]
  \includegraphics[height=2.5cm]{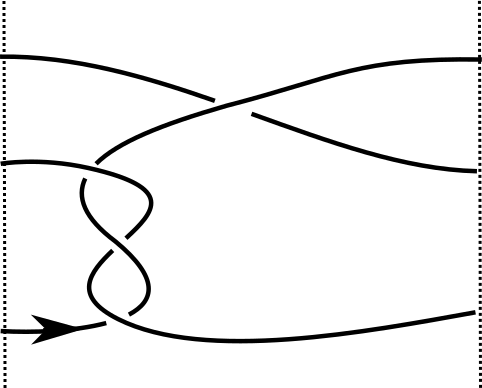} \quad \quad
  \includegraphics[height=2.5cm]{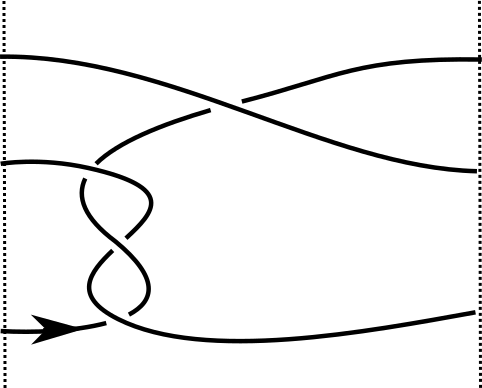} \quad \quad
  \includegraphics[height=2.5cm]{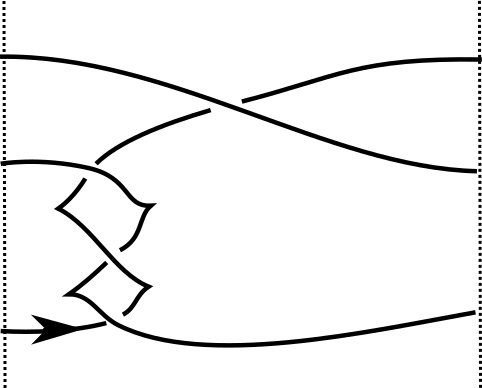}
  \label{fig:8a10}
  \caption{The patterns $-$L8a10 (left) and $R$ (center), and a Legendrian realization $\mathcal{R}$ of $R$  with $tb(\mathcal{R})=2$ and $rot(\mathcal{R})=0$ (right).}
\end{figure}
A single $(+)$ to $(-)$ crossing change takes $-$L8a10 to the pattern $R$, as depicted in Figure~\ref{fig:8a10}. It follows that 
$\tau(-\text{L8a10}(K)) \geq  \tau(R(K).$
We now argue as in Example~\ref{exl:8a8}, using the Legendrian realization $\mathcal{R}$ of $R$ on the right of Figure~\ref{fig:8a10} to say that there is a Legendrian diagram $\mathcal{R(K)}$ for $R(K)$ with 
\begin{align*}
tb(\mathcal{R(K)}) =3^2 \cdot 0 +2=2 \text{ and }
rot(\mathcal{R(K)}) = 3 \cdot 1+0=3.
\end{align*}
It therefore follows by Theorem~\ref{thm:tbtau} that 
\[\tau(-\text{L8a10}(K)) \geq \tau(R(K)) \geq \frac{2+ 3 +1}{2}=3.
\]
So $\text{L8a10}(K)$ and $(-\text{L8a10})(K)$ are not concordant and L8a10 does not induce a pseudo-homomorphism. 
\end{exl}

The two remaining patterns are L6a2 and L8a9, both of which can be easily seen to induce pseudo-homomorphisms, since each is slice and amphichiral. 
\begin{figure}[h!]
$\begin{array}{cc}
  \includegraphics[height=2.5cm]{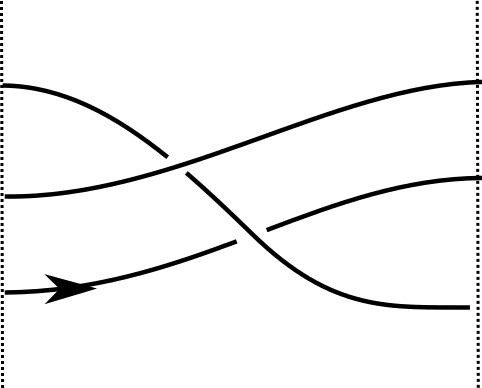} &   \includegraphics[height=2.5cm]{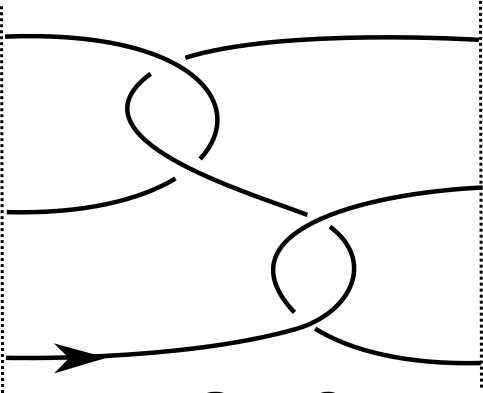}
  \end{array}$
  \caption{L6a2 (left) and L8a9 (right) induce pseudo-homomorphisms.}
  \label{fig:6a4}
\end{figure}

\begin{exl}[L6a2 induces a pseudo-homomorphism but not a homomorphism]\label{exl:L6a4}

 L6a4 induces a pseudo-homomorphism since L6a2$(U)$ is slice and the pattern L6a4 is isotopic to -L6a4. 
Now let $K$ be a knot with $\epsilon(K)=+1$.
Since a single crossing change $(+)$ to $(-)$ takes L6a2 to $C_{3,-1}$, we have that 
\[\tau(\text{L6a4}(K)) \leq \tau(C_{3,-1})(K)+1= (3\tau(K) -2) +1 = 3 \tau(K)-1< 3 \tau(K).\]
So Proposition~\ref{prop:tauhom} implies that L6a2 does not induce a homomorphism.

\end{exl}

It is not hard  to generalize L6a2 to produce patterns of each odd winding number which induce pseudo-homomorphisms yet not homomorphisms of $\C_s$.  This leads us to the following question about the existence of `non-standard pseudo-homomorphisms'.
We remark that this question also relates to whether all torsion elements of the concordance group are represented by negative amphichiral knots (Question 1.94, \cite{Kirby78}.) 

\begin{quest}\label{quest:pseudohom}
Let $P$ be a pattern inducing a pseudo-homomorphism on the concordance group. 
Must  $P$ be concordant in $(S^1 \times D^2) \times I$ to a pattern $Q$ with the property that $Q(-K)$ is isotopic to $-Q(K)$ for all $K$? \end{quest}

By work of Hartley~\cite{Hartley}, the winding number of a pattern with $P$ isotopic to $-P$ must either be 0 or odd. Since winding number is a concordance invariant of patterns,  an affirmative answer to Question~\ref{quest:pseudohom} would imply Conjecture~\ref{conj:nohom} for patterns of non-zero even winding number.

\begin{exl}[The pattern induced by L8a9]\label{exl:8a9}
We are left to consider $P$=L8a9.  Since $P=-P$ we see that this pattern induces a pseudo-homomorphism.
However, $P(K)$ and $K$ are very difficult to distinguish: in particular, since either a  $(+)$ to $(-)$ or a $(-)$ to $(+)$ crossing change takes $P$ to a core, we have that $\tau(P(K))= \tau(K)$ for all knots $K$.  One can also check that Rasmussen's $s$-invariant and many other smooth concordance invariants are similarly incapable of showing that $K$ and $P(K)$ are not concordant. 
However, it is straightforward to verify that $P_{+1}(U)$, the knot in $S^3$ obtained by doing a $+1$ twist along the meridian of the solid torus that $P$ lies within, is not even topologically slice and so that $P$ is not concordant to a core.

We are therefore left with the following questions: does  L8a9 act by the identity and, if not, does it induce a non-standard homomorphism of $\C_s$?
\end{exl}

\bibliography{research}

\def\cprime{$'$}
\begin{thebibliography}{DMOP16}

\bibitem[AIM19]{AIM19}
Problem list.
\newblock In {\em Smooth concordance classes of topologically slice knots},
  AIM, 2019.

\bibitem[CG86]{CG86}
A.~J. Casson and C.~McA. Gordon.
\newblock Cobordism of classical knots.
\newblock In {\em \`A la recherche de la topologie perdue}, volume~62 of {\em
  Progr. Math.}, pages 181--199. Birkh\"auser Boston, Boston, MA, 1986.
\newblock With an appendix by P. M. Gilmer.

\bibitem[CH18]{CH17}
Tim Cochran and Shelly Harvey.
\newblock The geometry of the knot concordance space.
\newblock {\em Algebr. Geom. Topol.}, 18(5):2509--2540, 2018.

\bibitem[CHL11]{Cochran-Harvey-Leidy:2011-02}
Tim~D. Cochran, Shelly Harvey, and Constance Leidy.
\newblock Primary decomposition and the fractal nature of knot concordance.
\newblock {\em Math. Ann.}, 351(2):443--508, 2011.

\bibitem[CL04]{Cha-Livingston:2004}
Jae~Choon Cha and Charles Livingston.
\newblock Knot signature functions are independent.
\newblock {\em Proc. Amer. Math. Soc.}, 132(9):2809--2816, 2004.

\bibitem[CL19]{LinkInfo}
Jae~Choon Cha and Charles Livingston.
\newblock Linkinfo: Table of knot invariants, May 31, 2019.

\bibitem[COT03]{Cochran-Orr-Teichner:1999-1}
Tim~D. Cochran, Kent~E. Orr, and Peter Teichner.
\newblock Knot concordance, {W}hitney towers and {$L\sp 2$}-signatures.
\newblock {\em Ann. of Math. (2)}, 157(2):433--519, 2003.

\bibitem[DMOP16]{DMOP16}
Christopher~W. Davis, Taylor Martin, Carolyn Otto, and JungHwan Park.
\newblock Every genus one algebraically slice knot is 1-solvable.
\newblock {\em Preprint, available at arXiv:1606.00479}, 2016.

\bibitem[FMPC19]{FellerMillerPC19}
Peter Feller, Allison~N. Miller, and Juanita Pinz{\'o}n-Caicedo.
\newblock A note on the topological slice genus of satellite knots.
\newblock {\em preprint, arxiv:1908.03760}, 2019.

\bibitem[GM95]{GM95}
Robert~E. Gompf and Katura Miyazaki.
\newblock Some well-disguised ribbon knots.
\newblock {\em Topology Appl.}, 64(2):117--131, 1995.

\bibitem[Gom86]{Gompf:1986-1}
Robert~E. Gompf.
\newblock Smooth concordance of topologically slice knots.
\newblock {\em Topology}, 25(3):353--373, 1986.

\bibitem[Har80]{Hartley}
Richard~I. Hartley.
\newblock Invertible amphicheiral knots.
\newblock {\em Math. Ann.}, 252(2):103--109, 1979/80.

\bibitem[Hed07]{Hedden:2007-1}
Matthew Hedden.
\newblock Knot {F}loer homology of {W}hitehead doubles.
\newblock {\em Geom. Topol.}, 11:2277--2338, 2007.

\bibitem[Hed09]{Hedden09}
Matthew Hedden.
\newblock On knot {F}loer homology and cabling. {II}.
\newblock {\em Int. Math. Res. Not. IMRN}, (12):2248--2274, 2009.

\bibitem[Hom14]{Hom2014-Tau-of-Cables}
Jennifer Hom.
\newblock Bordered {H}eegaard {F}loer homology and the tau-invariant of cable
  knots.
\newblock {\em J. Topol.}, 7(2):287--326, 2014.

\bibitem[HPC16]{HeddenPinzon}
Matthew Hedden and Pinz{\'o}n-Caicedo.
\newblock Satellites of infinite rank in the smooth concordance group.
\newblock {\em Preprint: available at arXiv:1809.04186}, 2016.

\bibitem[Kir78]{Kirby78}
Rob Kirby.
\newblock Problems in low dimensional manifold theory.
\newblock In {\em Algebraic and geometric topology ({P}roc. {S}ympos. {P}ure
  {M}ath., {S}tanford {U}niv., {S}tanford, {C}alif., 1976), {P}art 2}, Proc.
  Sympos. Pure Math., XXXII, pages 273--312. Amer. Math. Soc., Providence,
  R.I., 1978.

\bibitem[Lev16]{Lev16}
Adam~Simon Levine.
\newblock Nonsurjective satellite operators and piecewise-linear concordance.
\newblock {\em Forum Math. Sigma}, 4:e34, 47, 2016.

\bibitem[Lit84]{Lith84}
R.~A. Litherland.
\newblock Cobordism of satellite knots.
\newblock In {\em Four-manifold theory ({D}urham, {N}.{H}., 1982)}, volume~35
  of {\em Contemp. Math.}, pages 327--362. Amer. Math. Soc., Providence, RI,
  1984.

\bibitem[Mil19]{Mil19}
Allison~N. Miller.
\newblock Winding number {$m$} and {$-m$} patterns acting on concordance.
\newblock {\em Proc. Amer. Math. Soc.}, 147(6):2723--2731, 2019.

\bibitem[MP18]{MP17}
Allison~N. Miller and Lisa Piccirillo.
\newblock Knot traces and concordance.
\newblock {\em J. Topol.}, 11(1):201--220, 2018.

\bibitem[MPI16]{MPIM16}
Problem list.
\newblock In {\em Conference on 4-manifolds and knot concordance}, MPIM,
  Oct.~17-21 2016.

\bibitem[NT04]{NgTraynor}
Lenhard Ng and Lisa Traynor.
\newblock Legendrian solid-torus links.
\newblock {\em J. Symplectic Geom.}, 2(3):411--443, 2004.

\bibitem[OS03a]{OS03abs}
Peter Ozsv\'ath and Zolt\'an Szab\'o.
\newblock Absolutely graded {F}loer homologies and intersection forms for
  four-manifolds with boundary.
\newblock {\em Adv. Math.}, 173(2):179--261, 2003.

\bibitem[OS03b]{Ozsvath-Szabo:2003-1}
Peter Ozsv{\'a}th and Zolt{\'a}n Szab{\'o}.
\newblock Knot {F}loer homology and the four-ball genus.
\newblock {\em Geom. Topol.}, 7:615--639 (electronic), 2003.

\bibitem[Pic18]{Piccirillo-shakegenus}
Lisa Piccirillo.
\newblock Shake genus and slice genus.
\newblock {\em Preprint, available at arXiv:1803.09834}, 2018.

\bibitem[Pla04]{Plam-tbtau}
Olga Plamenevskaya.
\newblock Bounds for the {T}hurston-{B}ennequin number from {F}loer homology.
\newblock {\em Algebr. Geom. Topol.}, 4:399--406, 2004.

\bibitem[VC10]{VanCott10}
Cornelia~A. Van~Cott.
\newblock Ozsv\'{a}th-{S}zab\'{o} and {R}asmussen invariants of cable knots.
\newblock {\em Algebr. Geom. Topol.}, 10(2):825--836, 2010.

\bibitem[Yas15]{Yas15}
K.~Yasui.
\newblock Corks, exotic 4-manifolds, and knot concordance.
\newblock {\em Preprint, available at arXiv:1505.02551v3}, 2015.

\end{thebibliography}

\bibliographystyle{alpha}
\end{document}